\documentclass[sn-mathphys,Numbered]{sn-jnl} 

\usepackage{relsize}
\usepackage{lscape}

\usepackage{mathtools}
\usepackage{cite}
\usepackage{anyfontsize}
\usepackage{graphicx} 
\usepackage{multirow} 
\usepackage{amsmath,amssymb,amsfonts} 
\usepackage{amsthm} 
\usepackage{mathrsfs} 
\usepackage[title]{appendix} 
\usepackage{xcolor} 
\usepackage{float}
\usepackage{textcomp} 
\usepackage{manyfoot} 
\usepackage{booktabs} 
\usepackage{algorithm} 
\usepackage{algorithmicx} 
\usepackage{algpseudocode} 
\usepackage{listings}
\usepackage{array}
\usepackage{caption}
\usepackage{subcaption}
\usepackage[utf8]{inputenc}
\usepackage{booktabs}
\usepackage{amsfonts}
\usepackage{siunitx}
\usepackage{longtable}
\usepackage{eucal}
\usepackage{multirow}
\usepackage{enumerate}
\usepackage{lineno}
\usepackage{framed}
\usepackage{placeins}
\usepackage{cleveref} 
\newtheorem{theorem}{Theorem}[section]
\newtheorem{proposition}{Proposition}[section] 
\newtheorem{lemma}{Lemma}[section]

\newtheorem{remark}{Remark}[section] 
\newtheorem{assumption}{Assumption}[section] 
\newtheorem{definition}{Definition}[section]

\DeclareMathOperator{\argmin}{\textnormal{argmin}}
\begin{document}

\title[Article Title]{Nonlinear Conjugate Gradient Method for Multiobjective Optimization Problems of Interval-Valued Maps}

\author[1]{\fnm{Tapas} \sur{Mondal}}\email{tapas.ra.mat24@itbhu.ac.in}

\author*[1]{\fnm{Debdas} \sur{Ghosh}}\email{debdas.mat@iitbhu.ac.in}

\author[2]{\fnm{Jingxin} \sur{Liu}}\email{jingxinliu9@126.com}

\author[3]{\fnm{Jie} \sur{Li}}\email{lijie1@hljit.edu.cn}

\affil[1]{\orgdiv{Department of Mathematical Sciences}, \orgname{Indian Institute of Technology (BHU)}, \orgaddress{ \city{Varanasi}, \postcode{221005}, \state{Uttar Pradesh}, \country{India}}}

\affil[2]{\orgdiv{School of Mathematical and Physical Sciences, RI-IM·AI*}, \orgname{Chongqing University of Science and Technology}, \orgaddress{ \city{Chongqing}, \postcode{401331}, \country{China}}}

\affil[3]{\orgdiv{School of Civil Architectural Engineering}, \orgname{Heilongjiang Institute of Technology}, \orgaddress{ \city{Heilongjiang}, \postcode{150050}, \country{China}}}

\abstract{In this article, we propose an algorithm for the nonlinear conjugate gradient method to find a Pareto critical point of unconstrained multiobjective interval optimization problems. In this algorithm, we use the Wolfe line search procedure to find the step length. After defining the standard Wolfe conditions and the strong Wolfe conditions, we prove that there exists an interval of the step length that satisfies the standard Wolfe conditions and the strong Wolfe conditions. Further, to study the convergence analysis of our proposed algorithm, we derive the result related to the Zoutendijk condition. In the convergence analysis, first, we prove the global convergence property of our proposed algorithm for a general conjugate gradient algorithmic parameter. Further, we consider four variants of the conjugate gradient algorithmic parameter, such as Fletcher-Reeves, conjugate descent, Dai-Yuan, and modified Dai-Yuan. For each variant of the algorithmic parameter, we prove the global convergence results of our proposed algorithm. Finally, we test our algorithm on some test problems and make a performance profile.}



\maketitle
\section{Introduction}\label{Introduction}
Multiobjective optimization problems arise naturally in diverse scientific and engineering applications where several conflicting criteria must be optimized simultaneously. Classical formulations typically assume that each objective function is real-valued and can be evaluated precisely at any point in the decision space. However, in many real-world scenarios, the available data are imprecise, uncertain, or subject to measurement errors, leading to objective values that cannot be represented as exact values. Interval analysis offers a systematic framework for handling such uncertainty by modeling each objective value as an interval rather than a single number. Consequently, multiobjective optimization of {\it interval-valued mappings} (IVMs) has gained increasing interest, as it provides informative solutions under incomplete or vague information.

In parallel, gradient-based iterative methods continue to serve as the backbone of large-scale continuous optimization. Among them, nonlinear conjugate gradient methods are particularly attractive due to their simplicity, low memory requirements, and strong convergence properties for unconstrained problems. While nonlinear conjugate gradient methods have been extensively studied for single-objective optimization, extending them to multiobjective settings—especially when objective functions are interval-valued—poses several theoretical and computational challenges. These challenges include defining appropriate notions of descent directions, scalarization principles under interval uncertainty, and generalized optimality conditions that remain meaningful in the presence of multi-dimensional preferences and interval dominance relations.
\subsection{Literature Survey}
A substantial body of literature exists for multiobjective optimization, beginning with early scalarization and Pareto-based approaches (for instance, see \cite{ehrgott2005multicriteria,miettinen1999nonlinear,ghosh2014directed}). Gradient-based methods for multiobjective optimization problems have been developed using descent directions, criticality conditions, and various scalarization strategies (For instance, see \cite{fliege2000steepest,fliege2009newton,assuncao2021conditional,drummond2004projected,lapucci2023limited,mohammadi2024trust,polvaj2014quasi,perez2018nonlinear}).

 More recently, conjugate gradient-type approaches for deterministic multiobjective problems have been proposed, offering improved convergence speed and better performance on large-scale problems. Early progress in this direction is exemplified by the foundational work of Lucambio P{\'e}rez and Prudente \cite{perez2018nonlinear}, who provided one of the first systematic extensions of nonlinear conjugate gradient methods to vector optimization. Their contribution centered on generalizing classical Wolfe and strong Wolfe line search conditions to the multiobjective setting and proving a Zoutendijk-type convergence property for vector-valued functions. By constructing a suitable notion of descent grounded in the ordering of $\mathbb{R}^m$ induced by the nonnegative orthant, they demonstrated that several vector analogues of the 
Fletcher-Reeves (FR), 
Conjugate Descent (CD),  
Dai-Yuan (DY), Polak-Ribiere-Polyak (PRP), and Hestenes-Stiefel (HS)  formulas admit global convergence under the generalized Wolfe conditions. This work established key theoretical tools that subsequent research has relied on, especially the existence of acceptable step sizes and summability properties necessary for global convergence proofs. 
 
 Building on these foundations, researchers explored how specific conjugate gradient parameter choices behave in the vector context. A major development in this direction is the study of the Liu-Storey (LS) family of methods for multiobjective problems by Gon\c{c}alves et al. \cite{gonclaves2022study}. The authors observed that direct extensions of LS, while attractive for their performance in scalar optimization, may lose descent in the multiobjective setting. To overcome this issue, they proposed three LS-based variants incorporating modifications such as nonnegativity safeguarding of the conjugate gradient parameter, the use of strong Wolfe conditions, and the design of a new Armijo-type line search adapted to vector objectives. Each variant was accompanied by global convergence guarantees, and numerical experiments highlighted that properly integrating LS with line search strategies restores its effectiveness for vector optimization. Their results emphasize the delicate interplay between line search and parameter updates in multiobjective conjugate gradient algorithms.
 
 Further progress on stabilizing conjugate gradient parameters is seen in the paper by Pan et al. \cite{pan2025modified}, who developed an mDY method for vector optimization. Their approach ensures descent by appropriately adjusting the DY parameter and incorporating protective mechanisms inspired by the scalar theory. Similarly, Hu et al. \cite{Hu2024alternative} proposed an alternative extension of the Hager-Zhang (HZ) parameter to the vector case, addressing issues of instability and loss of conjugacy that can arise during straightforward extensions.
 
 Recent research has also explored hybrid conjugate gradient parameters that aim to blend desirable characteristics of classical formulas. Peng et al. \cite{peng2024novel} introduced a hybrid parameter combining the CD and DY updates. This adaptive choice enhances stability, guarantees descent under strong Wolfe conditions, and improves practical performance. Their multiobjective hybrid conjugate gradient method demonstrated competitive or superior numerical results compared with several existing conjugate gradient variants. 
 
 Ensuring positivity of the conjugate gradient parameter and sufficient descent conditions is critical for the convergence of multiobjective conjugate gradient methods. The comprehensive study by He et al. \cite{He2024family} presents a family of conjugate gradient formulas that simultaneously guarantee positivity and descent. Along similar lines, Chen et al. \cite{chen2026three} proposed a three-term conjugate gradient-type method possessing a uniform sufficient descent property. Their analysis demonstrates that the additional term plays a critical role in stabilizing the search direction, and numerical experiments confirm its effectiveness for vector optimization.
 
 Beyond classical conjugate gradient formulas, researchers have explored the incorporation of spectral scaling ideas into multiobjective conjugate gradient methods. The work of Elboulqe and Maghri \cite{Elboulqe2025explicit} presents an explicit spectral FR method for bi-criteria optimization. Spectral coefficients, inspired by Barzilai–Borwein step size estimators, are used to accelerate convergence. While originally designed for two objectives, the resulting approach highlights the potential benefits of spectral information and motivates further generalization to higher dimensions. Spectral-based conjugate gradient methods offer appealing numerical speed but often require safeguarding strategies to maintain descent, echoing trends observed in hybrid and positivity-guaranteeing methods.

\subsection{Motivation and Work Done}
In parallel, optimization techniques for IVMs have matured, supported by developments in generalized differentiability concepts such as $gH$-differentiability and order relations for interval vectors (see \cite{chauhan2021generalized,debnath2022generalized,ghosh2017newton,ghosh2017quasi,ghosh2022generalized,moore1966interval,stefanini2008generalization}). MIOPs have been explored under Pareto dominance extensions, interval preference relations, and interval scalarization functions, but iterative gradient-based solution methods for such settings remain comparatively limited. To address \emph{multiobjective interval optimization problems} (MIOPs), Upadhyay et al. \cite{upadhayay2024newton} proposed a Newton-type scheme and demonstrated its use in portfolio selection. Later, Upadhyay et al. \cite{upadhayay2024quasi} put forward a quasi-Newton variant for MIOPs. Their approach reformulates an MIOP into an equivalent vector problem with real-valued objectives by expressing each IVM as the sum of its lower and upper bound functions. The Newton and quasi-Newton procedures are then executed on this reformulated problem, which essentially reduces the interval problem to a simple real-valued case. Recently, Mondal and Ghosh \cite{mondal2025steepest} introduced the {\it steepest descent} (SD) method to find a Pareto critical point for an MIOP. The authors solved a strongly convex quadratic subproblem to compute a descent direction at a non-Pareto critical point. Further, introducing the steepest descent algorithm for an MIOP, they investigated the convergence properties of their proposed algorithm. 

Despite these advances, there is a noticeable gap in the literature concerning the development of efficient conjugate gradient-type algorithms specifically tailored for MIOPs. Existing steepest descent approaches for MIOPs often suffer from slow convergence, sensitivity to step length rules, and computational inefficiency in high-dimensional settings. On the other hand, conjugate gradient directions—designed to accelerate convergence by incorporating curvature information—have not been systematically investigated under interval uncertainty or multiobjective preference structures. Moreover, integrating conjugate gradient-type updates with generalized derivatives such as $gH$-gradients remains largely an open research direction.

The main purpose of this work is to extend the idea given in \cite{perez2018nonlinear} for MIOPs. The proposed approach constructs descent directions using interval extensions of classical conjugate gradient formulas such as FR, CD, DY, and mDY, adapted to the structure of multiobjective and interval dominance relations. The main contributions of this paper can be summarized as follows:
\begin{enumerate}
	\item[(i)] We define the standard Wolfe conditions and the strong Wolfe conditions for MIOPs. Further, we prove that there exists an interval of step length that satisfies the standard Wolfe conditions and the strong Wolfe conditions.
	\item[(ii)] Based on the Wolfe line search procedure, we develop an algorithm of the nonlinear conjugate gradient method to find a Pareto critical point of MIOPs.
	\item[(iii)] To study the convergence properties of our proposed algorithm, we derive the result related to the Zoutendijk condition.
	\item[(iv)] Without any restriction on the conjugate gradient parameter, we prove that the sequence generated by our proposed algorithm has global convergence.
	\item[(v)] We use four variants of the conjugate gradient parameter, such as FR, CD, DY, and mDY. For each variant, we prove the global convergence of the sequence generated by our proposed algorithm.
	\item[(vi)] Finally, we test our proposed algorithm on some test problems given in \cite{mondal2025steepest} and make a performance profile from the perspective of Dolan-Mor{\'e} \cite{dolan2002benchmarking}.
\end{enumerate}

\subsection{Delineation}
The organization of the article is outlined as follows. Section \ref{Preliminaries} reviews essential notions from interval analysis together with fundamental concepts related to MIOPs. Section \ref{Nonlinear Conjugate Gradient Method} introduces the nonlinear conjugate gradient framework adapted to the interval-based multiobjective setting. Section \ref{Convergence Analysis} presents the theoretical guarantees for the method. The computational behavior of the scheme is illustrated through benchmark examples in Section \ref{Numerical Experiments}. Lastly, Section \ref{Conclusion and Future Directions} summarizes the main findings and highlights possible avenues for future investigation.

\section{Preliminaries}\label{Preliminaries}
This section outlines the fundamental concepts of interval analysis and presents key definitions and preliminary results related to MIOPs. For ease of reference, we denote by ${\mathbb{R}}$ the set of all real numbers and by ${\mathbb{N}}$ the set of natural numbers.

\subsection{Interval Analysis}\label{interval analysis}
Let $\mathcal{I}(\mathbb{R})$ represent the family of all closed, bounded intervals on the real line.  
For a real scalar $\alpha \in \mathbb{R}$ and two intervals  
$A=[a_1,a_2]$ and $B=[b_1,b_2]$, the basic interval operations described in Moore {\normalfont\cite{moore1966interval}}—namely addition, subtraction, and scalar multiplication—will be denoted by $A\oplus B$, $A\ominus B$, and $\alpha \odot A$, respectively. These are defined as:
\begin{enumerate}
	\item[(i)] $A\oplus B := [\,a_1+b_1,\; a_2+b_2\,]$;
	\item[(ii)] $A\ominus B := [\,a_1-b_2,\; a_2-b_1\,]$;
	\item[(iii)] $\alpha \odot A :=
	\begin{cases}
		[\,\alpha a_1,\; \alpha a_2\,], & \text{if } \alpha \geq 0,\\[4pt]
		[\,\alpha a_2,\; \alpha a_1\,], & \text{if } \alpha < 0.
	\end{cases}
	$
\end{enumerate}

\medskip
\begin{definition}[$gH$-difference {\normalfont\cite{stefanini2008generalization}}]\label{gH difference definition}
	\normalfont
	Let $ A$, $ B$, and $ C$ be intervals belonging to $\mathcal{I}({\mathbb{R}})$. Whenever the relation $ A:= B\oplus  C$ holds, or equivalently $ B:=A\ominus  C$, the interval $ C$ is referred to as the $gH$-difference of $ A$ and $ B$. This quantity is denoted by $C:= A\ominus_{gH} B.$
	  For two specific intervals $ A:=\left[a_1,a_2\right]$ and $ B:=\left[b_1,b_2\right]$, the $gH$-difference takes the form \[ A\ominus_{gH}B:=\left[\min\left\{a_1-b_1,a_2-b_2\right\},\max\left\{a_1-b_1,a_2-b_2\right\}\right].\] 
\end{definition}

\medskip
A dominance rule for comparing intervals in a minimization setting—where lower values are regarded as preferable—is described below. 
\medskip

\begin{definition}[Dominance relation of intervals {\normalfont\cite{chauhan2021generalized}}]\label{Dominance definition}
	\normalfont
	Let $ A:=\left[a_1,a_2\right]$ and $ B:=\left[b_1,b_2\right]$ be two intervals taken from $ \mathcal{I}({\mathbb{R}})$.
	\begin{enumerate}
		\item[(i)]  When $a_1\geq b_1$ and $a_2\geq b_2$, interval $B$ is said to dominate interval $A$. This dominance is denoted by $A\succeq B$.
		\item[(ii)] If either $a_1 > b_1$ with $a_2 \geq b_2$, or $a_1 \geq b_1$ with $a_2 > b_2$, then $B$ is regarded as strictly dominating $A$, written $A \succ B$.
		\item[(iii)] When the conditions for dominance fail, we indicate this by $A \nsucceq B$.
		\item[(iv)] If the criteria for strict dominance do not hold, this is denoted by $A \nsucc B$.
		\item[(v)]  
		The intervals $A$ and $B$ are termed comparable whenever either $A \succeq B$ or $B \succeq A$ is true.
		
		\item[(vi)]  
		If neither interval dominates the other, then $A$ and $B$ are considered non-comparable.
	\end{enumerate}
\end{definition}

\medskip
The relation $ A\succeq  B$ may equally be expressed in reverse order as $ B\preceq A$. In the same manner, the notations $A\succ B$, $ A\nsucceq  B$, and $ A\nsucc  B$ can be written in their symmetric forms $ B\prec A$, $ B\npreceq  A$, and $ B\nprec  A$, respectively.

\medskip
\begin{definition}[Norm on ${\mathcal{I}({\mathbb{R}})}$ {\normalfont\cite{moore1966interval}}]\label{Norm on IR definition}
	\normalfont
	Let ${\mathbb{R}}_{+}:=\left\{x\in{\mathbb{R}}:x\geq0\right\}$. A norm on the interval space ${\mathcal{I}({\mathbb{R}})}$ is a mapping $\|\cdot\|_{{\mathcal{I}({\mathbb{R}})}}:{\mathcal{ I}({\mathbb{R}})}\to{\mathbb{R}}_+$ and for any interval $A:=\left[a_1,a_2\right]\in{\mathcal{I}({\mathbb{R}})}$, it is specified by \[\|A\|_{\mathcal{ I}({\mathbb{R}})}:=\max\:\left\{|a_1|,|a_2|\right\}.\]
\end{definition}

\medskip
\noindent
In what follows, the symbol $\|\cdot\|$ denotes the standard Euclidean norm on ${\mathbb{R}}^n$. Consider an IVM $ \Gamma:{\mathbb{R}}^n\to \mathcal{ I}({\mathbb{R}})$, written in the form \[\Gamma:=\left[\underline{\Upsilon},\overline{\Upsilon}\right],\] where $\underline{\Upsilon}:{\mathbb{R}}^n\to{\mathbb{R}}$ and $\overline{\Upsilon}:{\mathbb{R}}^n\to{\mathbb{R}}$ are ordinary real-valued functions satisfying 
\[\underline{\Upsilon}\left(x\right)\leq \overline{\Upsilon}\left(x\right) \text{ for every } x\in{\mathbb{R}}^n.\]
	 The functions $\underline{\Upsilon}$ and $\overline{\Upsilon}$ are referred to as the lower and upper bound functions associated with the IVM $\Gamma$.

\medskip
\begin{definition}[$gH$-continuity {\normalfont\cite{ghosh2017newton}}]\label{gH continuity of IVM definition}
	\normalfont
	An IVM $\Gamma:{\mathbb{R}}^n\rightarrow \mathcal{ I}({\mathbb{R}})$ is called $gH$-continuous at a point $ x^\star$ if \[\underset{\| h\|\to 0}{\lim}\left( \Gamma( x^\star+  h)\ominus_{gH} \Gamma( x^\star)\right)=\left[0,0\right].\]
\end{definition}

\medskip
\begin{definition}[$gH$-Lipschitz continuity {\normalfont\cite{ghosh2022generalized}}]\label{gH Lipschitz continuity of IVM definition}
	\normalfont
	An IVM $\Gamma:{\mathbb{R}}^n\rightarrow \mathcal{ I}({\mathbb{R}})$ is called $gH$-Lipschitz continuous if there exists a constant $K>0$ such that \[\|\Gamma(u)\ominus_{gH}\Gamma(v)\|_{{\mathcal{ I}({\mathbb{R}})}}\leq K\|u-v\| \text{ for all }u,v \in{\mathbb{R}}^n.\]
\end{definition}

\medskip
\begin{remark}
	\normalfont
	The behavior of an IVM and that of its endpoint functions are closely related. The endpoint functions $\underline{\Upsilon}$ and $\overline{\Upsilon}$ are continuous at a point $x^\star$ precisely when the IVM $\Gamma$ is $gH$-continuous at the same point and vice-versa (for instance, see \cite{ghosh2017newton}). Similarly, the endpoint functions $\underline{\Upsilon}$ and $\overline{\Upsilon}$ are Lipschitz continuous with some common constant if and only if $\Gamma$ is $gH$-Lipschitz continuous ( for instance, see {\normalfont\cite{ghosh2022generalized}}).
\end{remark}

\medskip
\begin{definition}[$gH$-derivative {\normalfont\cite{debnath2022generalized}}]\label{gH derivative of IVM definition}
	\normalfont
	Let $U \subseteq \mathbb{R}$ be an open interval and let $\Gamma:U \rightarrow \mathcal{I}(\mathbb{R})$ be an IVM.  
	We say that $\Gamma$ is $gH$-differentiable at a point $x^{\diamond}\in U$ when the limit
	\[
	\Gamma'(x^{\diamond}) :=
	\lim_{h \to 0} \frac{1}{h}\odot\left(\Gamma(x^{\diamond}+h)\ominus_{gH}\Gamma(x^{\diamond})\right) \text{ exists}.
	\]
  In such a case, this limit is called the $gH$-derivative of $\Gamma$ at $x^{\diamond}$. 
\end{definition}

\medskip
\begin{remark}
	\normalfont
	Let $U\subseteq \mathbb{R}$ be an open interval, and consider an IVM $\Gamma:U\to\mathcal{ I}\left(\mathbb{R}\right)$ written as $\Gamma:=\left[\underline{\Upsilon},\overline{\Upsilon}\right]$. If both endpoint functions $\underline{\Upsilon}\text{ and }\overline{\Upsilon}$ are differentiable at some point $x^\diamond\in U$, then $gH$-derivative of $\Gamma$ at $x^\diamond$ also exists. Furthermore, letting $\underline{\Upsilon}'\left(x^\diamond\right)\text{ and }\overline{\Upsilon}'\left(x^\diamond\right)$ denote these classical derivatives, one obtains \[\Gamma'\left(x^\diamond\right):=\left[\min\left\{\underline{\Upsilon}'\left(x^\diamond\right),\overline{\Upsilon}'\left(x^\diamond\right)\right\},\max\left\{\underline{\Upsilon}'\left(x^\diamond\right),\overline{\Upsilon}'\left(x^\diamond\right)\right\}\right].\]
\end{remark}

\medskip
\begin{definition}[$gH$-partial derivative {\normalfont\cite{debnath2022generalized}}]\label{gH partial derivative of IVM definition}
	\normalfont
	Let $U \subset \mathbb{R}^n$ be an open region and let 
	$\Gamma : U \to \mathcal{I}(\mathbb{R})$ be an IVM.  
	For a chosen point $\bar{x} := (\bar{x}_1,\ldots,\bar{x}_n)^\top$, define the 
	one--dimensional slice
	\[
	\Gamma_j(t) := \Gamma(\bar{x}_1,\ldots,\bar{x}_{j-1},\, t ,\, \bar{x}_{j+1},\ldots,\bar{x}_n).
	\]
	If $\Gamma_j$ is $gH$-differentiable at $t=\bar{x}_j$, the $gH$-partial derivative 
	of $\Gamma$ with respect to the $j$-th variable at $\bar{x}$ is 
	\[
	\partial_{j}^{gH}\Gamma(\bar{x}) := \big(\Gamma_j\big)'(\bar{x}_j), 
	\qquad j = 1,\ldots,n.
	\]
\end{definition}

\medskip
\begin{definition}[$gH$-gradient {\normalfont\cite{debnath2022generalized}}]\label{gH gradient of IVM definition}
	\normalfont
	Let $U\subseteq {\mathbb{R}}^n$ be an open domain and $\Gamma:U\to\mathcal{ I}\left(\mathbb{R}\right)$ be an IVM. For any point $\bar{x}\in U$, the $gH$-gradient of $\Gamma$ at $\bar{x}\in U$--denoted by $\nabla_{gH}H(\tilde{x})$--is the column vector formed by its $gH$-partial derivatives. Explicitly, \[\nabla_{gH}\Gamma(\tilde{x}):=\left(\partial_{1}^{gH}\Gamma(\bar{x}),\partial_{2}^{gH}\Gamma(\bar{x}),\ldots,\partial_{n}^{gH}\Gamma(\bar{x})\right)^\top.\] 
\end{definition}
Throughout the article, we use the notations $\nabla_{gH}$ and $\nabla$ to represent the $gH$-gradient and the real-valued gradient, respectively.

\medskip
\begin{definition}[Linear IVM {\normalfont\cite{ghosh2022generalized}}]\label{Linear IVM definition}
	\normalfont
	Let $V\subseteq{\mathbb{R}}^n$ be a linear subspace, and consider an IVM $\Gamma:V\to\mathcal{ I}\left(\mathbb{R}\right)$. We call $\Gamma$ is linear whenever its value at any vector $x:=(x_1,x_2,\ldots,x_n)^\top\in V$ can be expressed as \[\Gamma(x):=\bigoplus_{k=1}^{n}\Gamma(u_k)\odot x_k,\]
	where $\left\{u_k\right\}_{k=1}^{n}$ denotes the canonical basis of of ${\mathbb{R}}^n$, and the symbol `$\bigoplus_{j=1}^{n}$' represents repeated interval addition.
\end{definition}
\medskip
\begin{definition}[$gH$-differentiable IVM {\normalfont\cite{ghosh2022generalized}}]\label{$gH$-differentiabe IVM definition}
	\normalfont
	Let $U\subseteq {\mathbb{R}}^n$ be an open region, and $\Gamma:U\to\mathcal{ I}\left(\mathbb{R}\right)$ be an IVM. We say that $\Gamma$ is $gH$-differentiable at a point $x^\diamond\in U$ provided that one can find a linear IVM $L_{x^\diamond}:{\mathbb{R}}^n\to \mathcal{ I}\left({\mathbb{R}}\right)$, an auxiliary interval expression $R\left(\Gamma\left(x^\diamond;h\right)\right)$, and some $\epsilon>0$, such that for every increment $h\in{\mathbb{R}}^n$ with $\left\|h\right\|<\epsilon$ \[\Gamma\left(x^\diamond+h\right)\ominus_{gH}\Gamma\left(x^\diamond\right):=L_{x^\diamond}\left(h\right)\oplus\left\|h\right\|\odot R\left(\Gamma\left(x^\diamond;h\right)\right),\]
	and the remainder satisfies $R\left(\Gamma\left(x^\diamond;h\right)\right)\to[0,0]$ as $\left\|h\right\|\to 0.$
	
	\bigskip
	
	\noindent
	If $\Gamma:U\to\mathcal{ I}\left(\mathbb{R}\right)$ admits a $gH$-derivative at every point in $U$, we say that $\Gamma$ is $gH$-differentiable on the set $U$.
\end{definition}
\medskip
\begin{lemma}[ {\normalfont\cite{ghosh2022generalized}}]\label{linear IVM lemma}
	Let $U\subseteq {\mathbb{R}}^n$ be an open, and consider an IVM $\Gamma:U\to\mathcal{ I}\left(\mathbb{R}\right)$. Assume that $\Gamma$ is $gH$-differentiable at a point $x^\diamond$. Then there exists a $\delta>0$ such that for every $w\in{\mathbb{R}}^n$ and for every scalar $\alpha$ satisfying $\left|\alpha\right|\left\|w\right\|<\delta$, the following limit holds:
	\[\underset{\alpha\to0}{\lim}\tfrac{1}{\alpha}\odot\left(\Gamma\left(x^\diamond+\alpha w\right)\ominus_{gH}\Gamma\left(x^\diamond\right)\right):=L_{x^\diamond}\left(w\right),\]
	where $L_{x^\diamond}\left(w\right)$ is the linear IVM.
	\bigskip
	
	\noindent Furthermore, if the $gH$-gradient of $\Gamma$ exists at $x^\diamond$, written
	\[\nabla_{gH}\Gamma\left(x^\diamond\right):=\left(\partial_{1}^{gH}\Gamma\left(x^\diamond\right),\partial_{2}^{gH}\Gamma\left(x^\diamond\right),\ldots,\partial_{n}^{gH}\Gamma\left(x^\diamond\right)\right)^\top,\]   then the linear IVM $L_{x^\diamond}$ admits the representation
	\[L_{x^\diamond}\left(w\right):=\nabla_{gH} \Gamma\left(x^\diamond\right)^\top\odot w=\bigoplus_{j=1}^n  	\partial_{j}^{gH}\Gamma\left(x^\diamond\right)\odot w_j \text{ for all } w:=\left(w_1,w_2,\ldots,w_n\right)^\top \in {\mathbb{R}}^n.\]
\end{lemma}
\medskip

\begin{definition}[Convex IVM {\normalfont\cite{wu2007karush}}]\label{Convex IVM definition}
	\normalfont
	Let $D\subseteq {\mathbb{R}}^n$ be a convex region, and consider an IVM $\Gamma:D\to\mathcal{ I}\left(\mathbb{R}\right)$. We call $\Gamma$ convex if, for any pair of points $u,v\in D$ and every $\lambda\in[0,1]$, the following relation hold: \[ \Gamma\left(\lambda u+\left(1-\lambda\right)v\right)\preceq\lambda\odot \Gamma( u)\oplus(1-\lambda)\odot \Gamma(v).\]
\end{definition}
\medskip

\medskip
\subsection{Multiobjective Interval Optimization Problem}

Consider a mapping $G:{\mathbb{R}}^n\rightarrow\mathcal{ I}\left({\mathbb{R}}\right)^m$ composed of $m$ interval-valued components \[G:=\left(G_1,G_2,\ldots,G_m\right)^\top,\] 
where each $G_i:{\mathbb{R}}^n\rightarrow\mathcal{ I}\left({\mathbb{R}}\right)$ is assumed to be $gH$-continuously differentiable and expressed in the form \[G_i:=\left[\underline{G}_i,\overline{G}_i\right] \text{ for all }i=1,2,\ldots,m.\]
In this work, our goal is to solve the following MIOP:
\begin{align}\label{minG(x)}
	\underset{x\in{\mathbb{R}}^n}{\min}\: G(x).
\end{align}
Before proceeding with algorithmic developments, we recall the notions of weak Pareto minimizer, Pareto minimizer, and Pareto critical point associated with the MIOP \eqref{minG(x)}.

\medskip
\begin{definition}[Weak Pareto minimizer{\normalfont \cite{mondal2025steepest}}]\label{weakly Pareto optimal}
	\normalfont
	A vector $x^{\dagger}\in\mathbb{R}^n$ is called a weak Pareto minimizer of the MIOP \eqref{minG(x)} when no alternative vector $y\in\mathbb{R}^n$ can be found for which each component interval $G_i$ at $y$ strictly dominates upon that at $x^{\dagger}$; that is,
	\[
	\not\exists\, y\in\mathbb{R}^n \ \text{such that}\ 
	G_i(y)\prec G_i(x^{\dagger})
	\quad \text{for all } i=1,2,\ldots,m.
	\]
\end{definition} 

\medskip
\begin{definition}[Pareto minimizer {\normalfont \cite{mondal2025steepest}}]\label{Pareto optimal}
	\normalfont
A vector $x^{\dagger}\in\mathbb{R}^n$ is called a Pareto minimizer of the MIOP \eqref{minG(x)} when no alternative vector $y\in\mathbb{R}^n$ can be found for which each component interval $G_i$ at $y$ dominates upon that at $x^{\dagger}$; that is,
\[
\not\exists\, y\in\mathbb{R}^n \ \text{such that}\ 
G_i(y)\preceq G_i(x^{\dagger})
\quad \text{for all } i=1,2,\ldots,m.
\]
\end{definition} 
\medskip

\begin{definition}[Pareto critical point {\normalfont \cite{mondal2025steepest}}]\label{Pareto critical}
	\normalfont
A vector $x^{\diamond}\in\mathbb{R}^{n}$ is called a Pareto critical point of the MIOP \eqref{minG(x)} when no direction $v\in\mathbb{R}^{n}$ can be found such that 
\[ 
\nabla_{gH}G_i(x^{\diamond})^{\top}\odot v \prec [0,0]
\quad \text{for all } i=1,\ldots,m.
\]
\end{definition}
\medskip

\begin{lemma}
\emph{\cite{mondal2025steepest}}  \label{interrelation lemma of Pareto optimal and critical} 
	Assume that $G\in C_{gH}^1\left({\mathbb{R}}^n,I\left({\mathbb{R}}\right)^m\right)$, i.e., $G_1,G_2,\ldots,G_m$ are all $gH$-continuously differentiable IVMs.
	\begin{enumerate}
		\item[(i)]  
		Whenever a point $x^{\dagger}$ serves as a weak Pareto minimizer of the MIOP \eqref{minG(x)}, it must also satisfy the conditions required to be a Pareto critical point for the same MIOP \eqref{minG(x)}.
		
		\item[(ii)]  
		Suppose that the IVMs $G_{1},G_{2},\ldots,G_{m}$ are convex.  
		If a vector $x^{\dagger}\in\mathbb{R}^{n}$ fulfills the criterion of Pareto criticality for the MIOP \eqref{minG(x)}, then $x^{\dagger}$ is a weak Pareto minimizer of the MIOP \eqref{minG(x)}.
	\end{enumerate}
\end{lemma}
\medskip
\begin{definition}[Descent direction \normalfont \cite{mondal2025steepest}]\label{Descent direction}
	\normalfont
	A vector $v\in\mathbb{R}^n$ is called a descent direction for the MIOP \eqref{minG(x)} at a point $x\in\mathbb{R}^n$ if one can find some $\delta>0$ such that, for every $t\in(0,\delta)$,
	\[
	G_i(x+t v)\;\prec\; G_i(x)
	\qquad \text{for each } i=1,2,\ldots,m.
	\]
\end{definition}

\medskip
Observe that when each IVM $G_i$ is $gH$-differentiable and the vector $v$ serves as a descent direction for every IVM $G_i$ for all $i=1,2,\ldots,m$, Lemma \ref{linear IVM lemma} implies that one has
\begin{align*}
	\nabla_{gH} G_i(x)^\top\odot v:=\underset{t\to0}{\lim} \tfrac{1}{t}\odot\left[G_i(x+tv)\ominus_{gH} G_i(x)\right]\prec[0,0] \text{ for all } i=1,2,\ldots,m.
\end{align*}
To determine a direction $v\in{\mathbb{R}}^n$ that yields descent for the MIOP \eqref{minG(x)} at a point $x\in{\mathbb{R}}^n$, one seeks a vector $v$ satisfying 
\[\nabla_{gH} G_i(x)^\top\odot v\prec[0,0] \text{ for all }i=1,2,\ldots,m.\]
In order to characterize such directions at the given point $x$, we introduce, for each index $i=1,2,\ldots,m$, an IVM $g_x^i:{\mathbb{R}}^n\rightarrow \mathcal{ I}({\mathbb{R}})$, defined as follows.
\begin{align}\label{gxi}
	g_x^i(v):=\nabla_{gH} G_i(x)^\top\odot v.
\end{align}
Let the $j$-th entry of the $gH$-gradient of $G_i$ at $x$ be represented by
 $$\left[\underline{\nabla_{gH}G_i}(x)_j,\overline{\nabla_{gH}G_i}(x)_j\right]:=\left[\min\left\{\frac{\partial \underline{G}_i(x)}{\partial x_j},\frac{\partial \overline{G}_i(x)}{\partial x_j}\right\},\max\left\{\frac{\partial \underline{G}_i(x)}{\partial x_j},\frac{\partial \overline{G}_i(x)}{\partial x_j}\right\}\right].$$ 
 Let the IVM $g_x^i:{\mathbb{R}}^n\to \mathcal{ I}\left({\mathbb{R}}\right)$ be written in terms of its endpoint functions $\underline{g}_x^i$ and $\overline{g}_x^i$ and $\left|v\right|:=\left(\left|v_1\right|,\left|v_2\right|,\ldots,\left|v_n\right|\right)^\top$. In \cite{mondal2025steepest}, we see that the endpoint functions can be written compactly as 
\begin{equation}\label{gxi-lower-upper}
	\begin{rcases}
		\begin{aligned}
			\underline{g}_x^i(v)&:=\tfrac{1}{2}\left(\underline{\nabla_{gH} G_i}(x)+\overline{\nabla_{gH} G_i}(x)\right)^\top v-\tfrac{1}{2}\left({\overline{\nabla_{gH} G_i}(x)-\underline{\nabla_{gH} G_i}(x)}\right)^\top \left|v\right|\\
			\text{and }
			\overline{g}_x^i(v)&:=\tfrac{1}{2}\left(\underline{\nabla_{gH} G_i}(x)+\overline{\nabla_{gH} G_i}(x)\right)^\top v+\tfrac{1}{2}\left({\overline{\nabla_{gH} G_i}(x)-\underline{\nabla_{gH} G_i}(x)}\right)^\top \left|v\right|.
		\end{aligned}
	\end{rcases}
\end{equation}
\medskip
Define a mapping $\psi:{\mathbb{R}}^m\rightarrow {\mathbb{R}}$ by \[\psi(w):=\underset{i=1,2,\ldots,m}{\max}w_i.\] 
 For a specific point $x\in{\mathbb{R}}^n$, introduce a vector-valued function $\phi_x:{\mathbb{R}}^n\rightarrow {\mathbb{R}}^m$, defined by \[\phi_x(v):=\left(\overline{g}_x^1(v),\overline{g}_x^2(v),\ldots,\overline{g}_x^m(v)\right)^\top.\] The composition $\psi\circ \phi_x:{\mathbb{R}}^n\to {\mathbb{R}}$ therefore takes the form
\begin{align}\label{psi-phi def}
	\psi\circ \phi_x\left(v\right):=\underset{i=1,2,\ldots,m}{\max}\overline{g}_x^i\left(v\right).
\end{align}
We now focus on the following unconstrained optimization problem 
\begin{align}\label{unconstrained min}
	\underset{v\in{\mathbb{R}}^n}{\min}\left( \psi\circ\phi_x(v)+\tfrac{1}{2}\|v\|^2\right).
\end{align}
Note that the function being minimized in \eqref{unconstrained min} is strongly convex (for instance, see Lemma 3.1 in \cite{mondal2025steepest}). As a consequence, the problem \eqref{unconstrained min} admits a unique minimizer.
 Denote by $v(x)$ the unique point attaining this minimum, and let $\xi(x)$ represent the corresponding optimal value of \eqref{unconstrained min} at the given $x$, i.e., 
\begin{align}\label{v(x) and xi(x)} 
	v(x):=\underset{v\in{\mathbb{R}}^n}{\argmin}\: \left( \psi\circ\phi_x(v)+\tfrac{1}{2}\|v\|^2\right) \text{ and } \xi(x):=\underset{v\in{\mathbb{R}}^n}{\min} \left( \psi\circ\phi_x(v)+\tfrac{1}{2}\|v\|^2\right).
\end{align}
In \cite{mondal2025steepest}, we see that to determine the vector $v(x)$, the unconstrained optimization problem \eqref{unconstrained min}  is written in a constrained quadratic optimization problem as 
\begin{equation}\label{equivalent constrained problem}
	\begin{rcases}
		\begin{aligned}
			\underset{u,v\in{\mathbb{R}}^n,\tau\in{\mathbb{R}}}{\min}&\left(\tau+\tfrac{1}{2}\|v\|^2\right)\\
			\text{subject to } &\left(\underline{\nabla_{gH} G_i}(x)+\overline{\nabla_{gH} G_i}(x)\right)^\top v+\left({\overline{\nabla_{gH} G_i}(x)-\underline{\nabla_{gH} G_i}(x)}\right)^\top u\leq2\tau, i=1,2,\ldots m,\\
			&-u_j\leq v_j \leq u_j, j=1,2,\ldots n.
		\end{aligned}
	\end{rcases}
\end{equation}

\medskip
\noindent
The direction $v\left(x\right)$ obtained from solving \eqref{equivalent constrained problem} is closely connected to the notion of Pareto criticality for the MIOP \eqref{minG(x)}.
Moreover, the next statement provides a criterion for detecting such critical points by examining the quantities $\left\|v\left(x\right)\right\|$ and $\xi\left(x\right)$.
\medskip
\begin{lemma}
\emph{\cite{mondal2025steepest}} \label{descent direction finding lemma}
	Let $v(x)\in{\mathbb{R}}^n$ denote the unique minimizer of the problem \eqref{unconstrained min}, i.e., \[v(x):=\underset{v\in{\mathbb{R}}^n}{\argmin}\: \left( \psi\circ\phi_x(v)+\tfrac{1}{2}\|v\|^2\right),\]
	 and let $\xi\left(x\right)$ be the associated minimum value, that is, \[\xi(x):=\underset{v\in{\mathbb{R}}^n}{\min} \left( \psi\circ\phi_x(v)+\tfrac{1}{2}\|v\|^2\right).\] Under these definitions, the following properties are valid:
	\begin{enumerate}
		\item[$(i)$]  For every $x\in{\mathbb{R}}^n$, one always has $\xi(x)\leq 0$.
		\item[$(ii)$] Whenever $x$ is a Pareto critical point of the MIOP \eqref{minG(x)}, the minimizer collapses to the zero vector, i.e., $v(x)=0\in{\mathbb{R}}^n$, and the optimal value satisfies $\xi(x)= 0$.
		\item[$(iii)$] If $x$ fails to be a Pareto critical point of the MIOP \eqref{minG(x)}, then the strict inequality $\xi(x)< 0$ holds, and the corresponding minimizer $v(x)\ne 0\in{\mathbb{R}}^n$ becomes a descent direction of the objective function of the MIOP \eqref{minG(x)} at $x$.
		\item[$(iv)$] The mapping $x\mapsto v(x)$ remains bounded on every compact subset of ${\mathbb{R}}^n$.
		\item[$(v)$] The mapping $x\mapsto \xi(x)$ is continuous.
	\end{enumerate}
\end{lemma}
\medskip
\begin{remark}
	\normalfont
	Lemma \ref{descent direction finding lemma} shows that the condition $\xi(x)=0$ or $\left\|v(x)\right\|=0$ guarantees that $x$ is a Pareto critical point. Conversely, when $\xi(x)\ne0$ or $\left\|v(x)\right\|\ne0$, the vector $v(x)$ serves as a descent direction of the objective function of the MIOP \eqref{minG(x)} at $x$.
\end{remark}
\medskip
The following lemma is helpful in studying the convergence properties of our proposed algorithm.
\medskip
\begin{lemma}
\normalfont{\cite{perez2018nonlinear}} 
For any scalars $p,q$, and $\alpha\neq0$, we have 

\begin{enumerate}\label{inequality lemma}
		\item[(i)] $2pq\leq2\alpha^2p^2+\tfrac{q^2}{2\alpha^2}$ \text{ and }
		\item[(ii)] $\left(p+q\right)^2\leq\left(1+2\alpha^2\right)p^2+\left(1+\tfrac{1}{2\alpha^2}\right)q^2.$
	\end{enumerate}
\end{lemma}

\section{Nonlinear Conjugate Gradient Method}\label{Nonlinear Conjugate Gradient Method}
In this section, we develop the nonlinear conjugate gradient method to find a Pareto critical point of the MIOP \eqref{minG(x)}. First, we define the standard Wolfe conditions and the strong Wolfe conditions related to the line search procedure for the MIOP \eqref{minG(x)}.
\medskip
\begin{definition}[Standard and strong Wolfe conditions]\label{Standard and strong Wolfe conditions definition}
	\normalfont
	Let $d\in{\mathbb{R}}^n$ be a descent direction for all $G_i$ at $x$. Consider $0<\rho<\sigma<1$. We say that $t>0$ satisfies the standard Wolfe conditions if 
	\begin{equation}\label{Standard Wolfe condition 1}
		G_i\left(x+td\right)\preceq G_i\left(x\right)\oplus \left[\rho t, \rho t\right] \odot \psi\circ\phi_x\left(d\right) \text{ for all } i=1,2,\ldots,m
	\end{equation} 
	and 
	\begin{equation}\label{Standard Wolfe condition 2}
	\psi\circ\phi_{x+td}\left(d\right)\geq \sigma~ \psi\circ\phi_x\left(d\right).
	\end{equation} 
	We say that $t>0$ satisfies the strong Wolfe conditions if 
	\begin{equation}\label{Strong Wolfe condition 1}
		G_i\left(x+td\right)\preceq G_i\left(x\right)\oplus \left[\rho t, \rho t\right] \odot \psi\circ\phi_x\left(d\right) \text{ for all } i=1,2,\ldots,m
	\end{equation} 
	and 
	\begin{equation}\label{Strong Wolfe condition 2}
		\lvert\psi\circ\phi_{x+td}\left(d\right)\rvert\geq \sigma \lvert\psi\circ\phi_x\left(d\right)\rvert.
	\end{equation} 
\end{definition}
\medskip
The following result ensures that there exists an interval of the step length that satisfies the standard Wolfe conditions \eqref{Standard Wolfe condition 1} and \eqref{Standard Wolfe condition 2} and the strong Wolfe conditions \eqref{Strong Wolfe condition 1} and \eqref{Strong Wolfe condition 2}.
\medskip
\begin{proposition}\label{existence of an interval of steplength}
	Assume that $G_1, G_2, \ldots,G_m$ are $gH$-continuously differentiable IVMs and $d\in{\mathbb{R}}^n$ is a descent direction for all $G_i$ at $x$. If there exists $A_i\in \mathcal{ I}\left({\mathbb{R}}\right)$ such that $G_i\left(x+td\right)\succeq A_i$ for all $t>0$ and for all $i=1,2,\ldots,m$, then there exists an interval of step length satisfying the standard Wolfe conditions \eqref{Standard Wolfe condition 1} and \eqref{Standard Wolfe condition 2} and the strong Wolfe conditions \eqref{Strong Wolfe condition 1} and \eqref{Strong Wolfe condition 2}.
\end{proposition}
\medskip
\begin{proof}
	Let $q_i:\mathbb{R}\to\mathcal{I}\left({\mathbb{R}}\right)$ and $l_i:\mathbb{R}\to\mathcal{ I}\left({\mathbb{R}}\right)$ be two IVMs defined by \[q_i(t):=G_i\left(x+td\right)\text{ and } l_i(t):=G_i\left(x\right)\oplus \left[\rho t, \rho t\right] \odot \psi\circ\phi_x\left(d\right).\]
	Let $q_i:=\left[\underline{q}_i,\overline{q}_i\right]$ and $l_i:=\left[\underline{l}_i,\overline{l}_i\right]$. Note that $\overline{q}_i(0)=\overline{l}_i(0)$ and $\overline{q}_i(t)$ is bounded below for all $t>0$. Since $d\in{\mathbb{R}}^n$ is a descent direction for all $G_i$ at $x$, $\psi\circ\phi_x\left(d\right)<0$. Since $\rho\in\left(0,1\right)$, from the view of Lemma 3.1 in \cite{nocedal1999numerical} the line $\overline{l}_i(t)$ is unbounded below and must intersect $\overline{q}_i(t)$ at least one $t>0$. Therefore, for all $i=1,2,\ldots,m,$ there exists $t_i>0$ such that \[\overline{q}_i\left(t_i\right)=\overline{l}_i\left(t_i\right)\text{ and } \overline{q}_i\left(t\right)<\overline{l}_i\left(t\right)\text{ for all } t\in \left(0,t_i\right),\]
	i.e.,
	\begin{align*} & \overline{G}_i\left(x+t_id\right)=\overline{G}_i\left(x\right)+\rho t_i~\psi\circ\phi_x\left(d\right)\text{ and } \\  & \overline{G}_i\left(x+td\right)<\overline{G}_i\left(x\right)+\rho t~\psi\circ\phi_x\left(d\right)\text{ for all } t\in \left(0,t_i\right).
    \end{align*}
	Let $\overline{t}:=\min\left\{t_1,t_2,\ldots,t_m\right\}$. Then, we have \[ \overline{G}_i\left(x+td\right)\leq\overline{G}_i\left(x\right)+\rho t~\psi\circ\phi_x\left(d\right)\text{ for all } t\in \left(0,\overline{t}\right], i=1,2,\ldots,m.\]
	Similarly, it is easy to show that 
	\[ \underline{G}_i\left(x+td\right)\leq\underline{G}_i\left(x\right)+\rho t~\psi\circ\phi_x\left(d\right)\text{ for all } t\in \left(0,\underline{t}\right], i=1,2,\ldots,m.\]
	Let $t_{\min}:=\min\left\{\underline{t},\overline{t}\right\}$. Then, we have \[ G_i\left(x+td\right)\preceq G_i\left(x\right)\oplus\left[\rho t,\rho t\right]\odot\psi\circ\phi_x\left(d\right)\text{ for all } t\in \left(0,t_{\min}\right], i=1,2,\ldots,m.\]
	Note that $t_{\min}$ is equal to either $\underline{t}$ or $\overline{t}$. If $t_{\min}=\overline{t}$, then there exists an $i_0\in\left\{1,2,\ldots,m\right\}$ such that \[\overline{G}_{i_0}\left(x+t_{\min}d\right)=\overline{G}_{i_0}\left(x\right)+\rho t_{\min}~\psi\circ\phi_x\left(d\right).\]
	Similarly, if $t_{\min}=\underline{t}$, then there exists an $i_0\in\left\{1,2,\ldots,m\right\}$ such that \[\underline{G}_{i_0}\left(x+t_{\min}d\right)=\underline{G}_{i_0}\left(x\right)+\rho t_{\min}~\psi\circ\phi_x\left(d\right).\]
	Without loss of generality, we assume that $t_{\min}=\overline{t}$. Let us now consider a function $\Gamma:{\mathbb{R}}\to{\mathbb{R}}$ defined by $\Gamma(t):=\overline{G}_{i_0}\left(x+td\right)-\overline{G}_{i_0}\left(x\right)-\rho t~\psi\circ\phi_x\left(d\right).$ Note that $\Gamma(0)=\Gamma\left(t_{\min}\right).$ Therefore, by Rolle's theorem , there exists $\tilde{t}\in\left(0,t_{\min}\right)$ such that $\Gamma'(\tilde{t})=0.$ This implies that 
\begin{align*}
& \rho~\psi\circ\phi_x\left(d\right) \\ = &  \nabla\overline{G}_{i_0}\left(x+\tilde{t}d\right)^\top d \\
\leq & \max \left\{\nabla\overline{G}_{i_0}\left(x+\tilde{t}d\right)^\top d, \nabla\underline{G}_{i_0}\left(x+\tilde{t}d\right)^\top d\right\}\\
= & \tfrac{1}{2} \left(\nabla\underline{G}_{i_0}\left(x+\tilde{t}d\right)^\top d+\nabla\overline{G}_{i_0}\left(x+\tilde{t}d\right)^\top d+\left| \nabla\overline{G}_{i_0}\left(x+\tilde{t}d\right)^\top d-\nabla\underline{G}_{i_0}\left(x+\tilde{t}d\right)^\top d\right|\right)\\
\leq &  \tfrac{1}{2} \left[\left(\nabla\underline{G}_{i_0}\left(x+\tilde{t}d\right)+\nabla\overline{G}_{i_0}\left(x+\tilde{t}d\right)\right)^\top d+\left| \nabla\overline{G}_{i_0}\left(x+\tilde{t}d\right)-\nabla\underline{G}_{i_0}\left(x+\tilde{t}d\right)\right|^\top \left|d\right|\right]\\
= & \tfrac{1}{2} \biggl[\left(\underline{\nabla_{gH}G_{i_0}}\left(x+\tilde{t}d\right)+\overline{\nabla_{gH}G_{i_0}}\left(x+\tilde{t}d\right)\right)^\top d\\
			&\hspace{1cm}+\left| \overline{\nabla_{gH}G_{i_0}}\left(x+\tilde{t}d\right)-\underline{\nabla_{gH}G_{i_0}}\left(x+\tilde{t}d\right)\right|^\top \left|d\right|\biggr]\\
			 \leq & \underset{i=1,2,\ldots,m}{\max} \tfrac{1}{2} \biggl[\left(\underline{\nabla_{gH}G_{i}}\left(x+\tilde{t}d\right)+\overline{\nabla_{gH}G_{i}}\left(x+\tilde{t}d\right)\right)^\top d\\
			&\hspace{2.3cm}+\left| \overline{\nabla_{gH}G_{i}}\left(x+\tilde{t}d\right) - \underline{\nabla_{gH}G_{i}}\left(x+\tilde{t}d\right)\right|^\top \left|d\right|\biggr]\\
			= &  \psi\circ\phi_{x+\tilde{t}d}\left(d\right).
		\end{align*}
	By intermediate value theorem there exists $t^\star\in\left(0,\tilde{t}~\right]$ such that \[\rho~\psi\circ\phi_x\left(d\right)=\psi\circ\phi_{x+t^\star d}\left(d\right).\]
	Since $0<\rho<\sigma<1$ and $\psi\circ\phi_x\left(d\right)<0$, then we get \[\sigma~\psi\circ\phi_x\left(d\right)<\psi\circ\phi_{x+\tilde{t}d}\left(d\right)<0.\]
	Hence, there is a neighbourhood of $t^\star$ contained in $\left[0,t_{\min}\right]$ for which \eqref{Standard Wolfe condition 2} and \eqref{Strong Wolfe condition 2} hold. Therefore, the standard Wolfe conditions and the strong Wolfe conditions hold in this neighbourhood. This completes the proof.
\end{proof}
\medskip

We now provide a step-wise algorithm of the nonlinear conjugate gradient method for identifying Pareto critical points of the MIOP \eqref{minG(x)}.
\medskip
\begin{algorithm}[H]
	\caption{Nonlinear conjugate gradient method to find Pareto critical points of the MIOP \eqref{minG(x)} \label{Algorithm}}
	\begin{enumerate}[\bf{Step} 1 ]
		
		\item (Inputs)\\
		Provide all the $gH$-continuously differentiable IVMs $G_1,G_2,\ldots, G_m$. \\
		\item (Initialization)\\
	  Select the parameter values $\rho\in(0,1)$ and $\sigma\in (\rho,1)$, and an arbitrary point $x^0 \in {\mathbb{R}}^n$. Set the tolerance level $\epsilon > 0$. Initialize the iteration count as $k=0$.\\

		\item (Computation of $gH$-gradients at the point $x^k$)\\
	Compute $\nabla_{gH} G_i\left(x^k\right):=\left[\underline{\nabla_{gH}G_i}\left(x^k\right),\overline{\nabla_{gH}G_i}\left(x^k\right)\right]$ for all $i=1,2,\ldots,m$.
		\\
		\item (Computation of a steepest descent direction at the point $x^k$)\\
		Compute the optimal solution $v\left(x^k\right)$ and the optimal value $\xi\left(x^k\right)$ of the unconstrained minimization problem \eqref{unconstrained min}, i.e.,
		\[
		v\left(x^k\right):=\underset{v\in{\mathbb{R}}^n}{\argmin}\: \left( \psi\circ\phi_{x^k}(v)+\tfrac{1}{2}\|v\|^2\right)\text{ and }
		\xi\left(x^k\right):=\underset{v\in{\mathbb{R}}^n}{\min}\: \left( \psi\circ\phi_{x^k}(v)+\tfrac{1}{2}\|v\|^2\right).
		\]

		\item
		(Stopping criterion)\\
		If the condition $``\xi\left(x^k\right)>- \epsilon"$ is satisfied, terminate the procedure and report $x^k$ as a Pareto critical point.\\
		If the condition $``\xi\left(x^k\right)>- \epsilon"$ is not satisfied, proceed to {\bf Step 6}.\\

		\item 
		(Computation of conjugate direction)\\
		Compute \begin{align}\label{conjugate direction in algorithm}
			d^k:=\begin{cases}
				v\left(x^k\right) &\text{ if } k=0\\
				v\left(x^k\right)+\beta_k d^{k-1} &\text{ if } k\geq 1,
			\end{cases}
		\end{align}
		where $\beta_k$ is an algorithmic parameter.\\

		\item (Computation of step length)\\
	Find a step length $t_k>0$ such that 
	\begin{align}\label{Standard Wolfe condition in algorithm}
		\begin{cases}
					G_i\left(x^k+t_k d^k\right)\preceq G_i\left(x^k\right)\oplus \left[\rho t_k, \rho t_k\right] \odot \psi\circ\phi_{x^k}\left(d^k\right) \text{ for all } i=1,2,\ldots,m\\
				\psi\circ\phi_{x^k+t_k d^k}\left(d^k\right)\geq \sigma~ \psi\circ\phi_{x^k}\left(d^k\right),
		\end{cases}
	\end{align}
	or 
		\begin{align}\label{Strong Wolfe condition in algorithm}
		\begin{cases}
			G_i\left(x^k+t_k d^k\right)\preceq G_i\left(x^k\right)\oplus \left[\rho t_k, \rho t_k\right] \odot \psi\circ\phi_{x^k}\left(d^k\right) \text{ for all } i=1,2,\ldots,m\\
				\left|\psi\circ\phi_{x^k+t_k d^k}\left(d^k\right)\right|\geq \sigma \left|\psi\circ\phi_{x^k}\left(d^k\right)\right|.
		\end{cases}
	\end{align}
		\item (Update the iterative point)\\
		Update $x^{k+1} \gets x^k + t_k d^k$, $k \gets k + 1$, and go to {\bf Step 3}.
		
	\end{enumerate}
	
\end{algorithm}
The proper functioning of Algorithm \ref{Algorithm} relies mainly on {\bf Step 4} and {\bf Step 7}. Because the objective function in the minimization problem \eqref{unconstrained min} is strongly convex, the associated minimization problem admits a unique minimizer. This ensures that the direction $v\left(x^k\right)$ is well defined, which in turn confirms the validity of {\bf Step 4} within Algorithm \ref{Algorithm}. In addition, Proposition \ref{existence of an interval of steplength} ensures the existence of an interval of the step length $t_k$ that satisfies the standard Wolfe condition \eqref{Standard Wolfe condition in algorithm} or the strong Wolfe condition \eqref{Strong Wolfe condition in algorithm}. However, for the well-definedness of Algorithm \ref{Algorithm}, Proposition \ref{existence of an interval of steplength} requires $d^k$ to be a descent direction for all $G_i$ at $x^k$, which is equivalent to $\psi\circ\phi_{x^k}\left(d^k\right)<0.$ For certain parts of the convergence analysis, a stronger assumption will be needed
\begin{align}\label{sufficient descent condition}
	\psi\circ\phi_{x^k}\left(d^k\right)\leq c~\psi\circ\phi_{x^k}\left(v\left(x^k\right)\right) \text{ for some }c>0 \text{ and for all }k\geq0. 
\end{align}
The condition \eqref{sufficient descent condition} is known as the sufficient descent condition. Such a condition can be achieved through an appropriate line search strategy, assuming that $d^{k-1}$ acts as a descent direction for every $G_i$ at the point $x^{k-1}$. Under this assumption, Proposition \ref{existence of an interval of steplength} guarantees that a line search performed along $d^{k-1}$ using either the standard Wolfe conditions or the strong Wolfe conditions will yield the next iterate $x^k$. For the update direction, we have 
\[\psi\circ\phi_{x^k}\left(d^k\right)=\psi\circ\phi_{x^k}\left(v\left(x^k\right)+\beta_k d^{k-1}\right)\leq \psi\circ\phi_{x^k}\left(v\left(x^k\right)\right)+\beta_k~ \psi\circ\phi_{x^k}\left(d^{k-1}\right).\]
If the sequence$\left\{\beta_k\right\}$ remains bounded, then a suitable line search can be employed to sufficiently reduce $\left|\psi\circ\phi_{x^k}\left(d^{k-1}\right)\right|$, thereby ensuring that the descent condition \eqref{sufficient descent condition} holds. 

The following lemma provides a condition on $\beta_k$ that guarantees $d^k$ inherits the descent property..
\medskip
\begin{lemma}\label{descent property on dk proof lemma}
	Assume that in Algorithm \ref{Algorithm}, the sequence $\left\{\beta_k\right\}$ is defined so that it has the following property:
	\begin{align}\label{property 1}
		\beta_k\in \begin{cases}
			\left[0,\infty\right), & \text{ if } \psi\circ\phi_{x^k}\left(d^{k-1}\right)\leq0\\
			\left[0,-\tfrac{\psi\circ\phi_{x^k}\left(v\left(x^k\right)\right)}{\psi\circ\phi_{x^k}\left(d^{k-1}\right)}\right], & \text{ if } \psi\circ\phi_{x^k}\left(d^{k-1}\right)>0,
		\end{cases}
	\end{align}
	or 
	\begin{align}\label{property 2}
		\beta_k\in \begin{cases}
			\left[0,\infty\right), & \text{ if } \psi\circ\phi_{x^k}\left(d^{k-1}\right)\leq0\\
			\left[0,-\mu\tfrac{\psi\circ\phi_{x^k}\left(v\left(x^k\right)\right)}{\psi\circ\phi_{x^k}\left(d^{k-1}\right)}\right], & \text{ if } \psi\circ\phi_{x^k}\left(d^{k-1}\right)>0
		\end{cases}
	\end{align}
for some $\mu\in\left[0,1\right)$.
	When condition \eqref{property 1} is satisfied, the vector $d^k$ serves as a descent direction for every iteration $k$. On the other hand, if \eqref{property 2} is valid, then $d^k$ meets the sufficient descent condition \eqref{sufficient descent condition} with the constant $c=1-\mu$ for all $k$.
\end{lemma}
\medskip
\begin{proof}
	We will prove the second statement. Note that at $k=0$, $d^0=v\left(x^0\right)$ and $\psi\circ\phi_{x^0}\left(d^0\right)=\psi\circ\phi_{x^0}\left(v\left(x^0\right)\right)$. Since $\mu\in\left[0,1\right)$, for $c=1-\mu$, we get 
	\[\psi\circ\phi_{x^0}\left(d^0\right)\leq c~\psi\circ\phi_{x^0}\left(v\left(x^0\right)\right).\]
	Therefore, the statement is true for $k=0$. For $k\geq1$, we have 
	\[\psi\circ\phi_{x^k}\left(d^k\right)=\psi\circ\phi_{x^k}\left(v\left(x^k\right)+\beta_k d^{k-1}\right)\leq \psi\circ\phi_{x^k}\left(v\left(x^k\right)\right)+\beta_k~ \psi\circ\phi_{x^k}\left(d^{k-1}\right).\]
	Suppose that $\psi\circ\phi_{x^k}\left(d^{k-1}\right)\leq0.$ Then, we get $\beta_k\in\left[0,\infty\right)$. Therefore, we get \[\psi\circ\phi_{x^k}\left(d^k\right)\leq \psi\circ\phi_{x^k}\left(v\left(x^k\right)\right)\leq \left(1-\mu\right)~\psi\circ\phi_{x^k}\left(v\left(x^k\right)\right)=c~\psi\circ\phi_{x^k}\left(v\left(x^k\right)\right).\]
	Suppose that  $\psi\circ\phi_{x^k}\left(d^{k-1}\right)>0.$ Then from \eqref{property 2}, we get 
	\begin{align*}
		\psi\circ\phi_{x^k}\left(d^k\right)&\leq \psi\circ\phi_{x^k}\left(v\left(x^k\right)\right)+\beta_k~ \psi\circ\phi_{x^k}\left(d^{k-1}\right)\\
		&\leq \psi\circ\phi_{x^k}\left(v\left(x^k\right)\right)-\mu\tfrac{\psi\circ\phi_{x^k}\left(v\left(x^k\right)\right)}{\psi\circ\phi_{x^k}\left(d^{k-1}\right)}~\psi\circ\phi_{x^k}\left(d^{k-1}\right)\\
		&= \left(1-\mu\right)~\psi\circ\phi_{x^k}\left(v\left(x^k\right)\right)\\
		&=c~\psi\circ\phi_{x^k}\left(v\left(x^k\right)\right).
	\end{align*}
	This completes the proof of the second statement. The proof of the first statement is similar to the proof of the second statement.
\end{proof}
\section{Convergence Analysis}\label{Convergence Analysis}
In this section, we examine various selections of the parameter $\beta_k$, together with suitable line search procedures, that lead to Algorithm \ref{Algorithm} with global convergence guarantees. If Algorithm \ref{Algorithm} terminates after a finite number of iterations, then {\bf Step 5} directly ensures that the returned point is Pareto critical. For the convergence analysis, we suppose that starting from an initial vector $x^0$, Algorithm \ref{Algorithm} produces an infinite sequence $\left\{x^k\right\}$ for which $\xi\left(x^k\right)\neq 0$ and $\left\|v\left(x^k\right)\right\|\neq0$ for all $k\in\mathbb{N}$. To study the convergence results of Algorithm \ref{Algorithm}, we consider the following assumptions.
\medskip
\begin{framed}
\begin{assumption}\label{assumption 1}
	\normalfont
	For all $i=1,2,\ldots,m$, $\nabla_{gH}G_i$ is $gH$-Lipschitz continuous with Lipschitz constant $L_i$ on the level set \[L_0:=\left\{x\in{\mathbb{R}}^n:G_i\left(x\right)\preceq G_i\left(x^0\right)\text{ for all }i=1,2,\ldots,m\right\}.\]
\end{assumption}
\medskip
\begin{assumption}\label{assumption 2}
	\normalfont
	All monotonically nonincreasing sequences in $G\left(L_0\right)$ are bounded below, i.e., if the sequence $\left\{\left(G_1\left(x^k\right),G_2\left(x^k\right),\ldots,G_m\left(x^k\right)\right)\right\}_{k\in{\mathbb{N}}}\subset G\left(L_0\right)$ and for all $i=1,2,\ldots,m$, $G_i\left(x^{k+1}\right)\preceq G_i\left(x^k\right)$ for all $k\in{\mathbb{N}}$, then there exist $\mathcal{G}_1,\mathcal{G}_2,\ldots,\mathcal{G}_m$ such that for all $i=1,2,\ldots,m$, $\mathcal{G}_i\preceq G_i\left(x^k\right)$ for all $k\in{\mathbb{N}}$.
\end{assumption}
\medskip
\begin{assumption}\label{assumption 3}
	\normalfont
	The level set $L_0:=\left\{x\in{\mathbb{R}}^n:G_i\left(x\right)\preceq G_i\left(x^0\right)\text{ for all }i=1,2,\ldots,m\right\}$ is bounded.
\end{assumption}
\end{framed}
\medskip
It is important to note that these three assumptions can be viewed as natural generalizations of the conditions commonly used in scalar or vector optimization. Based on Assumption \ref{assumption 1} and Assumption \ref{assumption 2}, we demonstrate that the iterative scheme $x^{k+1}:=x^k+t_kd^k$ satisfies a property comparable to the classical Zoutendijk criterion, provided that the step length $t_k$ at each iterate $x^k$, $k=0,1,,2,\ldots,$ is selected according to the standard Wolfe conditions \eqref{Standard Wolfe condition in algorithm}. The following result will later serve as a key component in establishing the global convergence of Algorithm \ref{Algorithm}.
\medskip
\begin{proposition}\label{Zountendijk condition proposition}
	Assume that Assumptions \eqref{assumption 1} and \eqref{assumption 2} hold. Consider the iteration scheme $x^{k+1}:=x^k+t_kd^k, k\geq0$, where $d^k$ is the descent direction for all $G_i$ at $x^k$ and $t_k$ satisfies the standard Wolfe conditions \eqref{Standard Wolfe condition 1} and \eqref{Standard Wolfe condition 2}. Then 
	\begin{equation}\label{Zountendjk condition}
		\underset{k\geq0}{\mathlarger\sum}\tfrac{\left[\psi\circ\phi_{x^k}\left(d^k\right)\right]^2}{\left\|d^k\right\|^2}<+\infty.
	\end{equation}
\end{proposition}
\medskip
\begin{proof}
	By the standard Wolfe condition \eqref{Standard Wolfe condition 2}, we get 
	\begin{align}\label{proof of zoutendjk condition eq 1}
		\left(\sigma-1\right)~\psi\circ\phi_{x^k}\left(d^k\right)\leq\psi\circ\phi_{x^{k+1}}\left(d^k\right)-\psi\circ\phi_{x^k}\left(d^k\right)
	\end{align}
	We have 
	\begin{align*}
		& \psi\circ\phi_{x^{k+1}}\left(d^k\right)-\psi\circ\phi_{x^k}\left(d^k\right)\\
		=&\underset{i=1,2,\ldots,m}{\max}\tfrac{1}{2}\left[\left(\underline{\nabla_{gH}G_i}\left(x^{k+1}\right)+\overline{\nabla_{gH}G_i}\left(x^{k+1}\right)\right)^\top d^k+\left|\overline{\nabla_{gH}G_i}\left(x^{k+1}\right)-\underline{\nabla_{gH}G_i}\left(x^{k+1}\right)\right|^\top \left|d^k\right|\right]\\
		&-\underset{i=1,2,\ldots,m}{\max}\tfrac{1}{2}\left[\left(\underline{\nabla_{gH}G_i}\left(x^{k}\right)+\overline{\nabla_{gH}G_i}\left(x^{k}\right)\right)^\top d^k+\left|\overline{\nabla_{gH}G_i}\left(x^{k}\right)-\underline{\nabla_{gH}G_i}\left(x^{k}\right)\right|^\top \left|d^k\right|\right]\\
		\leq &\underset{i=1,2,\ldots,m}{\max}\tfrac{1}{2}\biggl[\left(\underline{\nabla_{gH}G_i}\left(x^{k+1}\right)+\overline{\nabla_{gH}G_i}\left(x^{k+1}\right)\right)^\top d^k+\left|\overline{\nabla_{gH}G_i}\left(x^{k+1}\right)-\underline{\nabla_{gH}G_i}\left(x^{k+1}\right)\right|^\top \left|d^k\right|\\
		&\hspace{2cm}-\left(\underline{\nabla_{gH}G_i}\left(x^{k}\right)+\overline{\nabla_{gH}G_i}\left(x^{k}\right)\right)^\top d^k-\left|\overline{\nabla_{gH}G_i}\left(x^{k}\right)-\underline{\nabla_{gH}G_i}\left(x^{k}\right)\right|^\top \left|d^k\right|\biggr]\\
		\leq &\underset{i=1,2,\ldots,m}{\max}\tfrac{1}{2}\biggl[\left(\underline{\nabla_{gH}G_i}\left(x^{k+1}\right)-\underline{\nabla_{gH}G_i}\left(x^{k}\right)+\overline{\nabla_{gH}G_i}\left(x^{k+1}\right)-\overline{\nabla_{gH}G_i}\left(x^{k}\right)\right)^\top d^k\\
		&\hspace{2cm}+\left|\overline{\nabla_{gH}G_i}\left(x^{k+1}\right)-\overline{\nabla_{gH}G_i}\left(x^{k}\right)+\underline{\nabla_{gH}G_i}\left(x^{k}\right)-\underline{\nabla_{gH}G_i}\left(x^{k+1}\right)\right|^\top \left|d^k\right|\biggr]\\
		\leq &\underset{i=1,2,\ldots,m}{\max}\tfrac{1}{2}\biggl[\left\|\underline{\nabla_{gH}G_i}\left(x^{k+1}\right)-\underline{\nabla_{gH}G_i}\left(x^{k}\right)+\overline{\nabla_{gH}G_i}\left(x^{k+1}\right)-\overline{\nabla_{gH}G_i}\left(x^{k}\right)\right\| \left\|d^k\right\|\\
		&\hspace{2cm}+\left\|\overline{\nabla_{gH}G_i}\left(x^{k+1}\right)-\overline{\nabla_{gH}G_i}\left(x^{k}\right)+\underline{\nabla_{gH}G_i}\left(x^{k}\right)-\underline{\nabla_{gH}G_i}\left(x^{k+1}\right)\right\| \left\|d^k\right\|\biggr]\\
		\leq &\underset{i=1,2,\ldots,m}{\max}\tfrac{1}{2}\left[2\left(\left\|\underline{\nabla_{gH}G_i}\left(x^{k+1}\right)-\underline{\nabla_{gH}G_i}\left(x^{k}\right)\right\|+\left\|\overline{\nabla_{gH}G_i}\left(x^{k+1}\right)-\overline{\nabla_{gH}G_i}\left(x^{k}\right)\right\|\right) \left\|d^k\right\|\right]\\
		=&\underset{i=1,2,\ldots,m}{\max}\left(\left\|\underline{\nabla_{gH}G_i}\left(x^{k+1}\right)-\underline{\nabla_{gH}G_i}\left(x^{k}\right)\right\|+\left\|\overline{\nabla_{gH}G_i}\left(x^{k+1}\right)-\overline{\nabla_{gH}G_i}\left(x^{k}\right)\right\|\right) \left\|d^k\right\|.
	\end{align*}
	Therefore, using Assumption \eqref{assumption 1}, we obtain 
	\begin{align}\label{proof of zoutendjk condition eq 2}
		\psi\circ\phi_{x^{k+1}}\left(d^k\right)-\psi\circ\phi_{x^k}\left(d^k\right)\leq \underset{i=1,2,\ldots,m}{\max} 2L_i t_k \left\|d^k\right\|^2=Lt_k \left\|d^k\right\|^2, \text{ where } L:=\underset{i=1,2,\ldots,m}{\max} 2L_i.
	\end{align}
	From \eqref{proof of zoutendjk condition eq 1} and \eqref{proof of zoutendjk condition eq 2}, we get 
	\[\left(\sigma-1\right)~\psi\circ\phi_{x^k}\left(d^k\right)\leq Lt_k \left\|d^k\right\|^2.\]
	Since $\psi\circ\phi_{x^k}\left(d^k\right)<0$ and $\left\|d^k\right\|\neq0$, we obtain 
	\begin{align}\label{proof of zoutendjk condition eq 3}
		\tfrac{\left[\psi\circ\phi_{x^k}\left(d^k\right)\right]^2}{\left\|d^k\right\|^2}\leq L t_k \tfrac{\psi\circ\phi_{x^k}\left(d^k\right)}{\sigma-1}
	\end{align}
	By the standard Wolfe condition \eqref{Standard Wolfe condition 1}, we have that the sequence $\left\{G_i\left(x^k\right)\right\}_{k\geq0}$ is monotonic decreasing and that \[G_i\left(x^{k+1}\right)\ominus_{gH}G_i\left(x^0\right)\preceq\left[\rho,\rho\right]\odot\sum_{l=0}^{k}t_l\psi\circ\phi_{x^l}\left(d^l\right)\text{ for all }k\geq0 \text{ and for all } i=1,2,\ldots,m.\]
	Then, by Assumption \eqref{assumption 2}, there exists $\mathcal{G}_i\in{\it I}\left({\mathbb{R}}\right)$ such that 
	\[\mathcal{G}_i\ominus_{gH}G_i\left(x^0\right)\preceq\left[\rho,\rho\right]\odot\sum_{l=0}^{k}t_l\psi\circ\phi_{x^l}\left(d^l\right)\text{ for all }k\geq0 \text{ and for all } i=1,2,\ldots,m.\]
	This implies that 
	\[\tfrac{\mathcal{G}_i\ominus_{gH}G_i\left(x^0\right)}{\sigma-1}\succeq\left[\rho,\rho\right]\odot\sum_{l=0}^{k}t_l\tfrac{\psi\circ\phi_{x^l}\left(d^l\right)}{\sigma-1}\text{ for all }k\geq0 \text{ and for all } i=1,2,\ldots,m,\]
	which implies that 
	$\underset{k\geq0}{\sum}t_k\tfrac{\psi\circ\phi_{x^k}\left(d^k\right)}{\sigma-1}<+\infty.$ 
	Hence, from \eqref{proof of zoutendjk condition eq 3}, we conclude that $\underset{k\geq0}{\mathlarger\sum}\tfrac{\left[\psi\circ\phi_{x^k}\left(d^k\right)\right]^2}{\left\|d^k\right\|^2}<+\infty,$
	and this completes the proof.
\end{proof}
\medskip
\begin{remark}
	\normalfont
	Condition \eqref{Zountendjk condition}, often referred to as the Zoutendijk condition, serves as a key tool for establishing convergence properties of many line search-based optimization schemes. In particular, when the SD method employs a line search that enforces the standard Wolfe conditions, it follows that
	 $\underset{k\to\infty}{\lim}\left\|v\left(x^k\right)\right\|=0.$ In fact, consider the iterative scheme $x^{k+1}:=x^k+t_kd^k$ with $d^k=v\left(x^k\right)$, and assume $t_k$ satisfies the standard Wolfe condition \eqref{Standard Wolfe condition 1} for all $k\geq0$. If $v\left(x^k\right)=0$ for some $k$, the result trivially holds. Since $\xi\left(x^k\right):=\psi\circ\phi_{x^k}\left(v\left(x^k\right)\right)+\tfrac{1}{2}\left\|v\left(x^k\right)\right\|^2<0,$ it follows that \[\tfrac{\left[\psi\circ\phi_{x^k}\left(d^k\right)\right]^2}{\left\|d^k\right\|^2}=\tfrac{\left[\psi\circ\phi_{x^k}\left(v\left(x^k\right)\right)\right]^2}{\left\|v\left(x^k\right)\right\|^2}\geq \tfrac{1}{4}\left\|v\left(x^k\right)\right\|^2.\]
	By adding this inequality across all iterations $k$, we obtain the Zoutendijk condition that 
	\[\underset{k\geq0}{\sum}\left\|v\left(x^k\right)\right\|^2<+\infty, \text{ which implies that } \underset{k\to \infty}{\lim}\left\|v\left(x^k\right)\right\|=0.\]  
\end{remark}
Before moving towards on the convergence results with respect to the specific choices of $\beta_k$, we give the main convergence result of Algorithm \ref{Algorithm}. Without imposing an explicit restriction on $\beta_k$, the main convergence result of Algorithm \ref{Algorithm} is as follows.

\medskip
\begin{theorem}\label{General convergence theorem}
	Consider the Algorithm \ref{Algorithm}. Assume that Assumptions \eqref{assumption 1} and \eqref{assumption 2} hold and that 
	\begin{align}\label{general convergence sum eq 1}
		\underset{k\geq0}{\sum}\tfrac{1}{\left\|d^k\right\|^2}=+\infty.
	\end{align}
\begin{enumerate}
	\item[(i)] If $d^k$ satisfies the sufficient descent condition \eqref{sufficient descent condition} and $t_k$ satisfies the standard Wolfe conditions \eqref{Standard Wolfe condition in algorithm}, then $\underset{k\to \infty}{\liminf}\left\|v\left(x^k\right)\right\|=0.$
		\item[(ii)] If $\beta_k\geq0$, $d^k$ is a descent direction for all $G_i$ at $x^k$, and $t_k$ satisfies the strong Wolfe conditions \eqref{Strong Wolfe condition in algorithm}, then $\underset{k\to \infty}{\liminf}\left\|v\left(x^k\right)\right\|=0.$
\end{enumerate}
\end{theorem}
\medskip
\begin{proof}
	If possible, assume that there exists a constant $\gamma>0$ such that 
	\[\left\|v\left(x^k\right)\right\|\geq \gamma\text{ for all } k\geq 0.\]
	We now prove item (i). Note that by \[\psi\circ\phi_{x^k}\left(v\left(x^k\right)\right)+\tfrac{1}{2}\left\|v\left(x^k\right)\right\|<0\]
	 and the sufficient descent condition \eqref{sufficient descent condition}, we get 
	 \[\tfrac{c^2\gamma^4}{4\left\|d^k\right\|^2}\leq \tfrac{c^2\left\|v\left(x^k\right)\right\|^4}{4\left\|d^k\right\|^2}\leq \tfrac{\left[c~\psi\circ\phi_{x^k}\left(v\left(x^k\right)\right)\right]^2}{\left\|d^k\right\|^2}\leq \tfrac{\left[c~\psi\circ\phi_{x^k}\left(d^k\right)\right]^2}{\left\|d^k\right\|^2}.\]
	  Since under the hypothesis of item (i), the Zoutendijk condition \eqref{Zountendjk condition} holds, we have a contradiction to \eqref{general convergence sum eq 1}. Hence, item (i) is proved.
	  
	  Now, we consider item (ii). We have \[-\beta_k d^{k-1}=-d^k+v\left(x^k\right).\]
	  Then, 
	  \begin{align*}
	  	0&\leq \beta_k^2\left\|d^{k-1}\right\|^2\leq \left(\left\|d^k\right\|+\left\|v\left(x^k\right)\right\|\right)^2\\
	  	&\leq \left(\left\|d^k\right\|+\left\|v\left(x^k\right)\right\|\right)^2+\left(\left\|d^k\right\|-\left\|v\left(x^k\right)\right\|\right)^2\\
	  	&\leq 2\left\|d^k\right\|^2+2\left\|v\left(x^k\right)\right\|^2.
	  \end{align*}
	  Therefore, we get 
	  \begin{align}\label{proof of general convergence thm eq. star1}
	  	\left\|d^k\right\|^2\geq -\left\|v\left(x^k\right)\right\|^2+\tfrac{\beta_k^2}{2}\left\|d^{k-1}\right\|^2.
	  \end{align}
	  On the other hand, by the definition of the conjugate direction $d^k$ given in \eqref{conjugate direction in algorithm} and the positiveness of $\beta_k$, we get 
	  \[\psi\circ\phi_{x^k}\left(d^k\right)\leq\psi\circ\phi_{x^k}\left(v\left(x^k\right)\right)+\beta_k~\psi\circ\phi_{x^k}\left(d^{k-1}\right).\]
	  Then,
	  \begin{align*}
	  	0&\leq -\psi\circ\phi_{x^k}\left(v\left(x^k\right)\right)\leq-\psi\circ\phi_{x^k}\left(d^k\right)+\beta_k~\psi\circ\phi_{x^k}\left(d^{k-1}\right)\\
	  	&\leq -\psi\circ\phi_{x^k}\left(d^k\right)-\sigma\beta_k~\psi\circ\phi_{x^{k-1}}\left(d^{k-1}\right) ~~~\left[\text{using the strong Wolfe condition }\eqref{Strong Wolfe condition 2}\right]. 
	  \end{align*}
	  Therefore, we get 
	  \begin{align*}
	  	&\left[\psi\circ\phi_{x^k}\left(v\left(x^k\right)\right)\right]^2\\
	  	\leq& \left[\psi\circ\phi_{x^k}\left(d^k\right)\right]^2+\sigma^2\beta_k^2\left[\psi\circ\phi_{x^{k-1}}\left(d^{k-1}\right)\right]^2+ 2\sigma~\psi\circ\phi_{x^k}\left(d^k\right)\beta_k~\psi\circ\phi_{x^{k-1}}\left(d^{k-1}\right).
	  \end{align*}
	   From item (i) of Lemma \ref{inequality lemma} with $\alpha=1$, $p=\sigma~\psi\circ\phi_{x^k}\left(d^k\right)$, and  $q=\beta_k~\psi\circ\phi_{x^{k-1}}\left(d^{k-1}\right)$, we get 
	    \begin{align*}
	   	&\left[\psi\circ\phi_{x^k}\left(v\left(x^k\right)\right)\right]^2\\
	   	\leq& \left[\psi\circ\phi_{x^k}\left(d^k\right)\right]^2+\sigma^2\beta_k^2\left[\psi\circ\phi_{x^{k-1}}\left(d^{k-1}\right)\right]^2+ 2\sigma^2\left[\psi\circ\phi_{x^k}\left(d^k\right)\right]^2+\tfrac{\beta_k^2}{2}\left[\psi\circ\phi_{x^{k-1}}\left(d^{k-1}\right)\right]^2\\
	   	\leq & \left(1+2\sigma^2\right)\left(\left[\psi\circ\phi_{x^k}\left(d^k\right)\right]^2+\tfrac{\beta_k^2}{2}\left[\psi\circ\phi_{x^{k-1}}\left(d^{k-1}\right)\right]^2\right).
	   \end{align*}
	   Since $d^k$ is a descent direction for all $G_i$ at $x^k$, we obtain 
	   \[\xi\left(x^k\right):=\psi\circ\phi_{x^k}\left(v\left(x^k\right)\right)+\tfrac{1}{2}\left\|v\left(x^k\right)\right\|^2<0.\]
	   Therefore, we get 
	   \begin{align}\label{proof of general convergence thm star2}
	   	\left[\psi\circ\phi_{x^k}\left(d^k\right)\right]^2+\tfrac{\beta_k^2}{2}\left[\psi\circ\phi_{x^{k-1}}\left(d^{k-1}\right)\right]^2\geq c_1\left\|v\left(x^k\right)\right\|^4, \text{ where } c_1:=\tfrac{1}{4\left(1+2\sigma^2\right)}.
	   \end{align}
	   Note that 
	   \begin{align*}
	   	& \tfrac{\left[\psi\circ\phi_{x^k}\left(d^k\right)\right]^2 }{\left\|d^k\right\|^2}+\tfrac{ \left[\psi\circ\phi_{x^{k-1}}\left(d^{k-1}\right)\right]^2}{\left\|d^{k-1}\right\|^2}\\
	   	=& \tfrac{1}{\left\|d^k\right\|^2}\left\{\left[\psi\circ\phi_{x^k}\left(d^k\right)\right]^2 +\tfrac{\left\|d^k\right\|^2}{\left\|d^{k-1}\right\|^2}\left[\psi\circ\phi_{x^{k-1}}\left(d^{k-1}\right)\right]^2\right\}\\
	   	\geq & \tfrac{1}{\left\|d^k\right\|^2}\left\{\left[\psi\circ\phi_{x^k}\left(d^k\right)\right]^2 +\left(\tfrac{\beta_k^2}{2}-\tfrac{\left\|v\left(x^k\right)\right\|^2}{\left\|d^{k-1}\right\|^2}\right)\left[\psi\circ\phi_{x^{k-1}}\left(d^{k-1}\right)\right]^2\right\}~~\left[\text{using }\eqref{proof of general convergence thm eq. star1}\right]\\
	   	\geq & \tfrac{1}{\left\|d^k\right\|^2}\left\{c_1\left\|v\left(x^k\right)\right\|^4 -\tfrac{\left\|v\left(x^k\right)\right\|^2}{\left\|d^{k-1}\right\|^2}\left[\psi\circ\phi_{x^{k-1}}\left(d^{k-1}\right)\right]^2\right\}~~\left[\text{using }\eqref{proof of general convergence thm star2}\right]\\
	   	= & \tfrac{\left\|v\left(x^k\right)\right\|^2}{\left\|d^k\right\|^2}\left\{c_1\left\|v\left(x^k\right)\right\|^2 -\tfrac{\left[\psi\circ\phi_{x^{k-1}}\left(d^{k-1}\right)\right]^2}{\left\|d^{k-1}\right\|^2}\right\}.
	   \end{align*}
	   Since under the hypothesis of item (ii), the Zoutendijk condition \eqref{Zountendjk condition} holds, and it implies that \[\tfrac{\left[\psi\circ\phi_{x^k}\left(d^k\right)\right]^2 }{\left\|d^k\right\|^2}\to 0\text{ as } k\to\infty.\]
	   Therefore, we get 
	   \begin{align}\label{prof of general convergence thm star3}
	   	 \tfrac{\left[\psi\circ\phi_{x^{k-1}}\left(d^{k-1}\right)\right]^2 }{\left\|d^{k-1}\right\|^2}\leq \tfrac{c_1}{2}\gamma^2\leq \tfrac{c_1}{2}\left\|v\left(x^k\right)\right\|^2.
	   \end{align}
	   So, using \eqref{prof of general convergence thm star3}, we get 
	   \begin{align*}
	   	 \tfrac{\left[\psi\circ\phi_{x^k}\left(d^k\right)\right]^2 }{\left\|d^k\right\|^2}+\tfrac{ \left[\psi\circ\phi_{x^{k-1}}\left(d^{k-1}\right)\right]^2}{\left\|d^{k-1}\right\|^2}
	   	 &\geq \tfrac{\left\|v\left(x^k\right)\right\|^2}{\left\|d^k\right\|^2}\left(c_1\left\|v\left(x^k\right)\right\|^2 -\tfrac{c_1}{2}\left\|v\left(x^k\right)\right\|^2\right)\\
	   	  &=\tfrac{c_1}{2}\tfrac{\left\|v\left(x^k\right)\right\|^4}{\left\|d^k\right\|^2}\geq \tfrac{c_1\gamma^4}{2}\tfrac{1}{\left\|d^k\right\|^2}.
	   \end{align*}
	   Therefore, we get 
	   \[\underset{k\geq0}{\sum}\tfrac{1}{\left\|d^k\right\|^2}<+\infty,\]
	   which contradicts to \eqref{general convergence sum eq 1}, and this completes the proof.
\end{proof}
\medskip
In the following, we analyze the convergence properties of Algorithm \ref{Algorithm} with respect to the specific choices of parameter $\beta_k$. 
\subsection{Fletcher-Reeves}
The FR parameter is given by 
\[\beta_k^{FR}:=\tfrac{\psi\circ\phi_{x^k}\left(v\left(x^k\right)\right)}{\psi\circ\phi_{x^{k-1}}\left(v\left(x^{k-1}\right)\right)}.\]
 In the following result, we prove under a suitable hypothesis that if the parameter $\beta_k$ in Algorithm \ref{Algorithm} is bounded in magnitude by any fraction of $\beta_k^{FR}$, then Algorithm \ref{Algorithm} has global convergence.

\medskip
\begin{theorem}\label{global convergence thm for FR}
	Consider the Algorithm \ref{Algorithm} and let $0\leq \delta<1$.
	\begin{enumerate}
		\item[(i)] Let Assumption \ref{assumption 1} and Assumption \ref{assumption 2} hold. If $\left|\beta_k\right|\leq \delta\beta_k^{FR}$, $d^k$ satisfies the sufficient descent condition \eqref{sufficient descent condition}, and $t_k$ satisfies the standard Wolfe conditions \eqref{Standard Wolfe condition in algorithm}, then $\underset{k\to\infty}{\liminf}\left\|v\left(x^k\right)\right\|=0.$
		\item[(ii)] Let Assumption \ref{assumption 1} and Assumption \ref{assumption 3} hold. If $0\leq\beta_k\leq \delta\beta_k^{FR}$, $d^k$ is a descent direction for all $G_i$ at $x^k$, and $t_k$ satisfies the strong Wolfe conditions \eqref{Strong Wolfe condition in algorithm}, then $\underset{k\to\infty}{\liminf}\left\|v\left(x^k\right)\right\|=0.$
	\end{enumerate} 
\end{theorem}
\medskip
\begin{proof}
	If possible, assume that there exists $\gamma>0$ such that 
		\[\left\|v\left(x^k\right)\right\|\geq \gamma\text{ for all } k\geq 0.\]
		By \eqref{conjugate direction in algorithm} and the item (ii) of Lemma \ref{inequality lemma} with $p=\left\|v\left(x^k\right)\right\|$, $q=\left|\beta_k\right|\left\|d^{k-1}\right\|$, and $\alpha=\tfrac{\delta}{\sqrt{2\left(1-\delta^2\right)}}$, we get 
		\[\left\|d^k\right\|^2\leq\left(\left\|v\left(x^k\right)\right\|+\left|\beta_k\right|\left\|d^{k-1}\right\|\right)^2\leq\tfrac{1}{1-\delta^2}\left\|v\left(x^k\right)\right\|^2+\tfrac{1}{\delta^2}\beta_k^2\left\|d^{k-1}\right\|^2.\]
		Since $\left|\beta_k\right|\leq \delta\beta_k^{FR}=\delta \tfrac{\psi\circ\phi_{x^k}\left(v\left(x^k\right)\right)}{\psi\circ\phi_{x^{k-1}}\left(v\left(x^{k-1}\right)\right)}$, we get 
		\[\tfrac{ \left\|d^k\right\|^2}{\left[\psi\circ\phi_{x^k}\left(v\left(x^k\right)\right)\right]^2}\leq\tfrac{1}{1-\delta^2} \tfrac{ \left\|v\left(x^k\right)\right\|^2}{\left[\psi\circ\phi_{x^k}\left(v\left(x^k\right)\right)\right]^2}+\tfrac{1}{\delta^2}\tfrac{ \beta_k^2\left\|d^{k-1}\right\|^2}{\left[\psi\circ\phi_{x^k}\left(v\left(x^k\right)\right)\right]^2}\leq\tfrac{1}{1-\delta^2} \tfrac{ \left\|v\left(x^k\right)\right\|^2}{\left[\psi\circ\phi_{x^k}\left(v\left(x^k\right)\right)\right]^2}+\tfrac{\left\|d^{k-1}\right\|^2}{\left[\psi\circ\phi_{x^{k-1}}\left(v\left(x^{k-1}\right)\right)\right]^2}.\]
		Due to the relation $0<\gamma^2\leq \left\|v\left(x^k\right)\right\|^2\leq-2\psi\circ\phi_{x^k}\left(v\left(x^k\right)\right)$, we get 
		\[\tfrac{ \left\|d^k\right\|^2}{\left[\psi\circ\phi_{x^k}\left(v\left(x^k\right)\right)\right]^2}\leq\tfrac{4}{\left(1-\delta^2\right)\gamma^2}+\tfrac{\left\|d^{k-1}\right\|^2}{\left[\psi\circ\phi_{x^{k-1}}\left(v\left(x^{k-1}\right)\right)\right]^2}.\]
		Applying repeatedly, we get 
		\begin{align}\label{proof of global convergence thm FR eq 1}
			\tfrac{ \left\|d^k\right\|^2}{\left[\psi\circ\phi_{x^k}\left(v\left(x^k\right)\right)\right]^2}\leq\tfrac{4}{\left(1-\delta^2\right)\gamma^2}k+\tfrac{\left\|d^{0}\right\|^2}{\left[\psi\circ\phi_{x^{0}}\left(v\left(x^{0}\right)\right)\right]^2} \leq \tfrac{4}{\left(1-\delta^2\right)\gamma^2}k+\tfrac{4}{\gamma^2}=\tfrac{4}{\gamma^2}\left(\tfrac{1}{1-\delta^2}k+1\right)
		\end{align}
		Therefore, we get 
	\begin{align}\label{proof of global convergence thm FR eq 2}
		\tfrac{\left[\psi\circ\phi_{x^k}\left(v\left(x^k\right)\right)\right]^2}{ \left\|d^k\right\|^2}\geq\tfrac{\gamma^2\left(1-\delta^2\right)}{4}\tfrac{1}{k+1-\delta^2}\geq\tfrac{\gamma^2\left(1-\delta^2\right)}{4}\tfrac{1}{k+1}.
	\end{align}
		Let us consider item (i). By the sufficient descent condition \eqref{sufficient descent condition}, we get 
		\[\underset{k\geq0}{\sum}\tfrac{\left[\psi\circ\phi_{x^k}\left(d^k\right)\right]^2}{\left\|d^k\right\|^2}\geq \underset{k\geq0}{\sum}\tfrac{c^2\left[\psi\circ\phi_{x^k}\left(v\left(x^k\right)\right)\right]^2}{\left\|d^k\right\|^2}\overset{\eqref{proof of global convergence thm FR eq 2}}{\geq}\tfrac{c^2\gamma^2\left(1-\delta^2\right)}{4}\underset{k\geq0}{\sum}\tfrac{1}{k+1}=+\infty.\]
		Therefore, we have a contradiction because under the hypothesis of item (i), the Zoutendijk condition \eqref{Zountendjk condition} holds. Hence, item (i) is proved.
		
		Let us now show item (ii). Since $d^k$ is a descent direction for all $G_i$ at $x^k$, we claim that under Assumption \ref{assumption 3}, the sequence $\left\{\psi\circ\phi_{x^k}\left(v\left(x^k\right)\right)\right\}$ is bounded. For all $x\in L_0$ and for all $i=1,2,\ldots,m$, we have
		\begin{align*}
			~&~ -\tfrac{1}{2}\left\lVert\underline{\nabla_{gH} G_i}\left(x^k\right)+\overline{\nabla_{gH} G_i}\left(x^k\right)\right\rVert\|v\left(x^k\right)\|\\
			\leq ~&~ \tfrac{1}{2}\left(\underline{\nabla_{gH} G_i}\left(x^k\right)+\overline{\nabla_{gH} G_i}\left(x^k\right)\right)^\top v\left(x^k\right)\quad\left[\text{Using Cauchy Schwarz inequality}\right]\\
			\leq ~&~ \tfrac{1}{2}\left(\underline{\nabla_{gH} G_i}\left(x^k\right)+\overline{\nabla_{gH} G_i}\left(x^k\right)\right)^\top v\left(x^k\right)+\tfrac{1}{2}\left({\overline{\nabla_{gH} G_i}\left(x^k\right)-\underline{\nabla_{gH} G_i}\left(x^k\right)}\right)^\top \left| v\left(x^k\right)\right|\\
			=~&~  \overline{g}_{x^k}^i\left( v\left(x^k\right)\right)
			\leq   \underset{i=1,2,\ldots,m}{\max}\overline{g}_{x^k}^i\left( v\left(x^k\right)\right)
			=  \psi\circ\phi_{x^k}\left( v\left(x^k\right)\right)<0.
		\end{align*}
		From item (iv) of Lemma \ref{descent direction finding lemma}, $v\left(x^k\right)$ is bounded in $L_0$. Since $\nabla_{gH}\underline{G_i}$ and $\nabla_{gH}\overline{G_i}$ are continuous on $L_0$ for all $i=1,2,\ldots,m$, there exists $M>0$ such that \[0>\psi\circ\phi_{x^k}\left( v\left(x^k\right)\right)\geq-M,\]
		which implies that \[\left[ \psi\circ\phi_{x^k}\left( v\left(x^k\right)\right)\right]^2\leq M^2.\]
		Therefore, by \eqref{proof of global convergence thm FR eq 1}, we get 
		\[\left\|d^k\right\|^2\leq ak+b,\text{ where } a:=\tfrac{4M^2}{\gamma^2\left(1-\delta^2\right)}>0\text{ and } b:=\tfrac{4M^2}{\gamma^2}>0.\]
		Then \[\underset{k\geq0}{\sum}\tfrac{1}{\left\|d^k\right\|^2}=+\infty,\]
		which implies by item (ii) of Theorem \ref{General convergence theorem} that $\underset{k\to\infty}{\liminf}\left\|v\left(x^k\right)\right\|=0.$
		
\end{proof}

\subsection{Conjugate Descent}
Next,  we study the convergence results of Algorithm \ref{Algorithm} with respect to the CD parameter, which is given by 
\[\beta_k^{CD}:=\tfrac{\psi\circ\phi_{x^k}\left(v\left(x^k\right)\right)}{\psi\circ\phi_{x^{k-1}}\left(d^{k-1}\right)}.\]
In the following lemma, we show that the sequence $\left\{d^k\right\}$ generated by Algorithm \ref{Algorithm} satisfies the sufficient descent condition \eqref{sufficient descent condition} under certain conditions.
\medskip
\begin{lemma}\label{CD parameter lemma}
Consider the Algorithm \ref{Algorithm} with $0\leq \beta_k\leq \beta_k^{CD}$ and suppose that $t_k$ satisfies the strong Wolfe conditions \eqref{Strong Wolfe condition in algorithm}. Then, $d^k$ satisfies the sufficient descent condition \eqref{sufficient descent condition} with $c:=1-\sigma$.
\end{lemma}
\medskip
\begin{proof}
	 We will prove this result by the mathematical induction method. For $k=0$, we have \[d^0=v\left(x^0\right) \text{ and }\psi\circ\phi_{x^0}\left(v\left(x^0\right)\right)<0.\]
	 Since $0<\sigma<1$, we have 
	 \[\psi\circ\phi_{x^0}\left(d^0\right)=\psi\circ\phi_{x^0}\left(v\left(x^0\right)\right)< \left(1-\sigma\right)\psi\circ\phi_{x^0}\left(v\left(x^0\right)\right)=c~\psi\circ\phi_{x^0}\left(v\left(x^0\right)\right).\]
	 For some $k\geq1$, we assume that 
	 \[\psi\circ\phi_{x^{k-1}}\left(d^{k-1}\right)\leq \left(1-\sigma\right)\psi\circ\phi_{x^{k-1}}\left(v\left(x^{k-1}\right)\right)<0.\]
	 Therefore, $\beta_k^{CD}>0$ and $\beta_k$ is well-defined. By the definition of the conjugate direction $d^k$ given in \eqref{conjugate direction in algorithm} and the positiveness of $\beta_k$, we get 
	 \[\psi\circ\phi_{x^k}\left(d^k\right)\leq\psi\circ\phi_{x^k}\left(v\left(x^k\right)\right)+\beta_k~\psi\circ\phi_{x^k}\left(d^{k-1}\right).\]
	 By the strong Wolfe condition \eqref{Strong Wolfe condition 2}, we get 
	 \begin{align*}
	 	\psi\circ\phi_{x^k}\left(d^k\right)
	 	&\leq \psi\circ\phi_{x^k}\left(v\left(x^k\right)\right)-\sigma\beta_k~\psi\circ\phi_{x^{k-1}}\left(d^{k-1}\right)\\
	 	&\leq \psi\circ\phi_{x^k}\left(v\left(x^k\right)\right)-\sigma\beta_k^{CD}~\psi\circ\phi_{x^{k-1}}\left(d^{k-1}\right)\\
	 	&= \psi\circ\phi_{x^k}\left(v\left(x^k\right)\right)-\sigma\tfrac{\psi\circ\phi_{x^k}\left(v\left(x^k\right)\right)}{\psi\circ\phi_{x^{k-1}}\left(d^{k-1}\right)}~\psi\circ\phi_{x^{k-1}}\left(d^{k-1}\right)\\ 
	 	&=\left(1-\sigma\right)\psi\circ\phi_{x^k}\left(v\left(x^k\right)\right).
	 \end{align*}
	 Hence, the proof is completed.
\end{proof}
	\medskip
\begin{remark}
	\normalfont
	In Lemma \ref{CD parameter lemma}, we show that the sequence $\left\{d^k\right\}$ generated by Algorithm \ref{Algorithm} satisfies the sufficient descent condition \eqref{sufficient descent condition} if $0\leq \beta_k\leq \beta_k^{CD}$ and the step length $t_k$ satisfies the strong Wolfe conditions \eqref{Strong Wolfe condition in algorithm}. 
\end{remark}
\medskip
In the following, we show the global convergence result of Algorithm \ref{Algorithm} under the consideration that $\beta_k\geq0$, bounded above by an appropriate fraction of the CD parameter $\beta_k^{CD}$, and the step length $t_k$ satisfies the strong Wolfe conditions \eqref{Strong Wolfe condition in algorithm}.
\medskip
\begin{theorem}\label{global convergence thm for CD parameter}
	Let Assumption \ref{assumption 1} and Assumption \ref{assumption 2} hold. Consider the Algorithm \ref{Algorithm} with \[\beta_k:=\eta\beta_k^{CD}, \text{ where }0\leq \eta< 1-\sigma.\] 
	If $t_k$ satisfies the strong Wolfe conditions \eqref{Strong Wolfe condition in algorithm}, then $\underset{k\to\infty}{\liminf}\left\|v\left(x^k\right)\right\|=0.$ 
\end{theorem}
\medskip
\begin{proof}
	From Lemma \ref{CD parameter lemma}, note that $d^k$ satisfies the sufficient descent conditions \eqref{sufficient descent condition} with $c:=1-\sigma$ for all $k\geq0$. So, we have 
	\begin{align*}
	\psi\circ\phi_{x^{k-1}}\left(d^{k-1}\right)\leq \left(1-\sigma\right)\psi\circ\phi_{x^{k-1}}\left(v\left(x^{k-1}\right)\right)<0. 
	\end{align*}
	Since $\psi\circ\phi_{x^{k}}\left(v\left(x^{k}\right)\right)<0$, we obtain that 
	\begin{align}\label{proof of global convergence thm CD eq 1}
		 \tfrac{\psi\circ\phi_{x^{k}}\left(v\left(x^{k}\right)\right)}{\psi\circ\phi_{x^{k-1}}\left(d^{k-1}\right)}\leq \tfrac{\psi\circ\phi_{x^{k}}\left(v\left(x^{k}\right)\right)}{\left(1-\sigma\right)\psi\circ\phi_{x^{k-1}}\left(v\left(x^{k-1}\right)\right)}=\tfrac{\beta_k^{FR}}{1-\sigma}.
	\end{align}
	Therefore, by \eqref{proof of global convergence thm CD eq 1}, we get that \[0\leq \beta_k=\eta\beta_k^{CD}=\tfrac{\eta}{1-\sigma}\left(1-\sigma\right)\tfrac{\psi\circ\phi_{x^{k}}\left(v\left(x^{k}\right)\right)}{\psi\circ\phi_{x^{k-1}}\left(d^{k-1}\right)}\leq\tfrac{\eta}{1-\sigma}\beta_k^{FR}.\]
	Since $0\leq \tfrac{\eta}{1-\sigma}<1$, from Theorem \ref{global convergence thm for FR}, we conclude that $\underset{k\to\infty}{\liminf}\left\|v\left(x^k\right)\right\|=0,$ and this completes the proof.
\end{proof}

\subsection{Dai-Yuan}
We derive the convergence result of Algorithm \ref{Algorithm} with respect to the DY parameter, which is given by
\[\beta_k^{DY}:=\tfrac{-\psi\circ\phi_{x^k}\left(v\left(x^k\right)\right)}{\psi\circ\phi_{x^{k}}\left(d^{k-1}\right)-\psi\circ\phi_{x^{k-1}}\left(d^{k-1}\right)}.\]
In the following lemma, we prove that the sequence $\left\{d^k\right\}$ generated by Algorithm \ref{Algorithm} satisfies the sufficient descent condition \eqref{sufficient descent condition} under some reasonable conditions.
\medskip
\begin{lemma}\label{DY lemma}
Consider the Algorithm \ref{Algorithm} with $0\leq\beta_k\leq \beta_k^{DY}$ and suppose that the step length $t_k$ satisfies the strong Wolfe conditions \eqref{Strong Wolfe condition in algorithm}. Then, $d^k$ satisfies the sufficient descent condition \eqref{sufficient descent condition} with $c:=\tfrac{1}{1+\sigma}$	 
\end{lemma}
\medskip
\begin{proof}
	 We will prove this result by the mathematical induction method. For $k=0$, we get 
	 \[d^0=v\left(x^0\right) \text{ and }\psi\circ\phi_{x^0}\left(v\left(x^0\right)\right)<0.\]
	 Since $0<\sigma<1$, we have 
	 \[\psi\circ\phi_{x^0}\left(d^0\right)=\psi\circ\phi_{x^0}\left(v\left(x^0\right)\right)< \tfrac{1}{\left(1+\sigma\right)}\psi\circ\phi_{x^0}\left(v\left(x^0\right)\right).\]
	 For some $k\geq1$, we assume that 
	  \[\psi\circ\phi_{x^{k-1}}\left(d^{k-1}\right)\leq \tfrac{1}{\left(1+\sigma\right)}\psi\circ\phi_{x^{k-1}}\left(v\left(x^{k-1}\right)\right)<0.\]
	  Since $\sigma<1$ and $\psi\circ\phi_{x^{k-1}}\left(d^{k-1}\right)<0$, by the strong Wolfe conditions \eqref{Strong Wolfe condition in algorithm}, we obtain that 
	  \[\psi\circ\phi_{x^{k}}\left(d^{k-1}\right)\geq \sigma~\psi\circ\phi_{x^{k-1}}\left(d^{k-1}\right)>\psi\circ\phi_{x^{k-1}}\left(d^{k-1}\right).\]
	  Therefore, $\beta_k^{DY}>0$ and $\beta_k^{DY}$ is well-defined. By the definition of the conjugate direction $d^k$ given in \eqref{conjugate direction in algorithm} and the positiveness of $\beta_k$, we get 
	  \begin{align}\label{proof of DY lemma eq 1}
	  	\psi\circ\phi_{x^k}\left(d^k\right)\leq\psi\circ\phi_{x^k}\left(v\left(x^k\right)\right)+\beta_k~\psi\circ\phi_{x^k}\left(d^{k-1}\right).
	  \end{align}
	  If $\psi\circ\phi_{x^k}\left(d^{k-1}\right)\leq0$, then we obtain that  \[\psi\circ\phi_{x^k}\left(d^k\right)\leq\psi\circ\phi_{x^k}\left(v\left(x^k\right)\right)<\tfrac{1}{1+\sigma}\psi\circ\phi_{x^k}\left(v\left(x^k\right)\right).\]
	  Therefore, if $\psi\circ\phi_{x^k}\left(d^{k-1}\right)\leq0$, then the result is true. Now, we suppose that $\psi\circ\phi_{x^k}\left(d^{k-1}\right)>0$. Since $0\leq\beta_k\leq \beta_k^{DY}$, from \eqref{proof of DY lemma eq 1}, we obtain that 
	   \begin{align}\label{proof of DY lemma eq 2}
	   	\psi\circ\phi_{x^k}\left(d^k\right)\leq \psi\circ\phi_{x^k}\left(v\left(x^k\right)\right)+\beta_k^{DY}~\psi\circ\phi_{x^k}\left(d^{k-1}\right).
	   \end{align}
	   Define $l_k:=\tfrac{\psi\circ\phi_{x^k}\left(d^{k-1}\right)}{\psi\circ\phi_{x^{k-1}}\left(d^{k-1}\right)}.$ By the strong Wolfe conditions \eqref{Strong Wolfe condition in algorithm}, we have $l_k\in\left[-\sigma,\sigma\right].$ From \eqref{proof of DY lemma eq 2}, we get 
	   \begin{align*}
	   	 \psi\circ\phi_{x^k}\left(d^k\right)&\leq \psi\circ\phi_{x^k}\left(v\left(x^k\right)\right)-\tfrac{\psi\circ\phi_{x^k}\left(v\left(x^k\right)\right)~\psi\circ\phi_{x^k}\left(d^{k-1}\right)}{\psi\circ\phi_{x^{k}}\left(d^{k-1}\right)-\psi\circ\phi_{x^{k-1}}\left(d^{k-1}\right)}\\
	   	 &\leq \psi\circ\phi_{x^k}\left(v\left(x^k\right)\right)\left(1-\tfrac{l_k}{l_k-1}\right)\\
	   	 &=\tfrac{1}{1-l_k}\psi\circ\phi_{x^k}\left(v\left(x^k\right)\right)\\
	   	 &\leq \tfrac{1}{1+\sigma}\psi\circ\phi_{x^k}\left(v\left(x^k\right)\right).
	   \end{align*}
	   Hence, the proof is completed.
\end{proof}
\medskip
In the following, we show the global convergence result of Algorithm \ref{Algorithm} with respect to the DY parameter.

\medskip
\begin{theorem}\label{global convergence thm for DY parameter}
	Let Assumption \ref{assumption 1} and Assumption \ref{assumption 2} hold. Consider the Algorithm \ref{Algorithm} with \[\beta_k:=\eta \beta_k^{DY}, \text{ where }0\leq \eta<\tfrac{1-\sigma}{1+\sigma}.\] 
	If the step length $t_k$ satisfies the strong Wolfe conditions \eqref{Strong Wolfe condition in algorithm}, then $\underset{k\to\infty}{\liminf}\left\|v\left(x^k\right)\right\|=0.$
\end{theorem}
\medskip
\begin{proof}
	 From Lemma \ref{DY lemma}, it is noted that $d^k$ satisfies the sufficient descent condition \eqref{sufficient descent condition} with $c:=\tfrac{1}{1+\sigma}$ for all $k\geq0$. So, we have 
	 \begin{align*}
	 \psi\circ\phi_{x^{k-1}}\left(d^{k-1}\right)\leq \tfrac{1}{\left(1+\sigma\right)}\psi\circ\phi_{x^{k-1}}\left(v\left(x^{k-1}\right)\right)<0.	 
	 \end{align*}
	 Since $\sigma<1$, we obtain that 
	 \begin{align}\label{proof of DY theorem eq 1}
	 	\left(\sigma-1\right)\psi\circ\phi_{x^{k-1}}\left(d^{k-1}\right)\geq \tfrac{\left(\sigma-1\right)}{\left(1+\sigma\right)}~\psi\circ\phi_{x^{k-1}}\left(v\left(x^{k-1}\right)\right)>0.	 
	 \end{align}
	 By the strong Wolfe conditions \eqref{Strong Wolfe condition in algorithm}, we get 
	 \begin{align*}
	 	\psi\circ\phi_{x^{k}}\left(d^{k-1}\right)-\psi\circ\phi_{x^{k-1}}\left(d^{k-1}\right)&\geq\left(\sigma-1\right)\psi\circ\phi_{x^{k-1}}\left(d^{k-1}\right)\\
	 	&\overset{\eqref{proof of DY theorem eq 1}}{\geq}\tfrac{\left(\sigma-1\right)}{\left(1+\sigma\right)}~\psi\circ\phi_{x^{k-1}}\left(v\left(x^{k-1}\right)\right)>0.\end{align*}
	 	Therefore, we get 
	 	\begin{align}\label{proof of DY thm eq 2}
	 		\tfrac{-\psi\circ\phi_{x^{k-1}}\left(v\left(x^{k-1}\right)\right)}{\psi\circ\phi_{x^{k}}\left(d^{k-1}\right)-\psi\circ\phi_{x^{k-1}}\left(d^{k-1}\right)}\leq \tfrac{1+\sigma}{1-\sigma}.
	 	\end{align}
	 	Define $\delta:=\eta\tfrac{1+\sigma}{1-\sigma}.$ Then, we obtain that 
	 	\begin{align*}
	 		 \beta_k&=\eta\beta_k^{DY}=\delta\tfrac{1-\sigma}{1+\sigma}\left[\tfrac{-\psi\circ\phi_{x^k}\left(v\left(x^k\right)\right)}{\psi\circ\phi_{x^{k}}\left(d^{k-1}\right)-\psi\circ\phi_{x^{k-1}}\left(d^{k-1}\right)}\right]\\
	 		 &=\delta\tfrac{1-\sigma}{1+\sigma}\left[\tfrac{-\psi\circ\phi_{x^k}\left(v\left(x^k\right)\right)}{-\psi\circ\phi_{x^{k-1}}\left(v\left(x^{k-1}\right)\right)}\right] \left[\tfrac{-\psi\circ\phi_{x^{k-1}}\left(v\left(x^{k-1}\right)\right)}{\psi\circ\phi_{x^{k}}\left(d^{k-1}\right)-\psi\circ\phi_{x^{k-1}}\left(d^{k-1}\right)}\right]\\
	 		 &\overset{\eqref{proof of DY thm eq 2}}{\leq}\delta \tfrac{\psi\circ\phi_{x^k}\left(v\left(x^k\right)\right)}{\psi\circ\phi_{x^{k-1}}\left(v\left(x^{k-1}\right)\right)}=\delta \beta_k^{FR}.
	 	\end{align*}
	 	Since $0\leq\delta<1$, from Theorem \ref{global convergence thm for FR}, we conclude that $\underset{k\to\infty}{\liminf}\left\|v\left(x^k\right)\right\|=0,$ and this completes the proof.
\end{proof}
\medskip
Theorem \ref{global convergence thm for DY parameter} states the global convergence of the Algorithm \ref{Algorithm} for an appropriate fraction of DY parameter. Next, we show that the global convergence of Algorithm \ref{Algorithm} with the slight modification of DY parameter. The modified DY parameter is given by 
\[\beta_k^{mDY}:=\tfrac{-\psi\circ\phi_{x^k}\left(v\left(x^k\right)\right)}{\psi\circ\phi_{x^{k}}\left(d^{k-1}\right)-\zeta\psi\circ\phi_{x^{k-1}}\left(d^{k-1}\right)},\text{ where }\zeta>1.\]
\medskip
\begin{lemma}\label{mDY parameter lemma}
	Consider the Algorithm \ref{Algorithm} with $0\leq \beta_k\leq \beta_k^{mDY}$ and suppose that $t_k$ satisfies the strong Wolfe conditions \eqref{Strong Wolfe condition in algorithm}. Then, $d^k$ satisfies the sufficient descent condition \eqref{sufficient descent condition} with $c:=\tfrac{\zeta}{\zeta+\sigma}$.
\end{lemma}
\medskip
\begin{proof}
	 The proof is similar to the proof of Lemma \ref{DY lemma}.
\end{proof}
\medskip
\begin{theorem}\label{mDY thm}
Let Assumption \ref{assumption 1} and Assumption \ref{assumption 2} hold. Consider the Algorithm \ref{Algorithm} with 
\[\beta_k:=\beta_k^{mDY}.\]	 
If the step length $t_k$ satisfies the strong Wolfe conditions \eqref{Strong Wolfe condition in algorithm}, then $\underset{k\to\infty}{\liminf}\left\|v\left(x^k\right)\right\|=0.$
\end{theorem}
\medskip
\begin{proof}
	 From Lemma \ref{mDY parameter lemma}, we observe that $\beta_k^{mDY}>0$ and $d^k$ satisfies the sufficient descent condition \eqref{sufficient descent condition} with $c:=\tfrac{\zeta}{\zeta+\sigma}$ for all $k\geq0$.
	 So, we have 
	 \begin{align}\label{proof of mDY theorem eq 1}
	  \psi\circ\phi_{x^{k}}\left(d^{k}\right)\leq \tfrac{\zeta}{\zeta+\sigma}\psi\circ\phi_{x^{k}}\left(v\left(x^{k}\right)\right)<0.	 
	 \end{align}
	 We have that 
	 \begin{align}\label{proof of mDY theorem eq 2}
	 	&\beta_k^{mDY}=\tfrac{-\psi\circ\phi_{x^k}\left(v\left(x^k\right)\right)}{\psi\circ\phi_{x^{k}}\left(d^{k-1}\right)-\zeta\psi\circ\phi_{x^{k-1}}\left(d^{k-1}\right)}\nonumber\\	
	 	\implies& \psi\circ\phi_{x^k}\left(v\left(x^k\right)\right)+\beta_k^{mDY}~\psi\circ\phi_{x^k}\left(d^{k-1}\right)=\zeta\beta_k^{mDY}\psi\circ\phi_{x^{k-1}}\left(d^{k-1}\right).
	 \end{align}
	By the definition of the conjugate direction $d^k$ given in \eqref{conjugate direction in algorithm} and the positiveness of $\beta_k^{mDY}$, we get 
	\begin{align*}
		\psi\circ\phi_{x^k}\left(d^k\right)\leq\psi\circ\phi_{x^k}\left(v\left(x^k\right)\right)+\beta_k^{mDY}~\psi\circ\phi_{x^k}\left(d^{k-1}\right)\overset{\eqref{proof of mDY theorem eq 2}}{=}\zeta\beta_k^{mDY}\psi\circ\phi_{x^{k-1}}\left(d^{k-1}\right),
	\end{align*}
	which implies that 
	\begin{align}\label{proof of mDY theorem eq 3}
		\beta_k^{mDY}\leq \tfrac{\psi\circ\phi_{x^k}\left(d^k\right)}{\zeta~\psi\circ\phi_{x^{k-1}}\left(d^{k-1}\right)} .
	\end{align}
		
If possible, assume that there exists $\gamma>0$ such that 
	\[\left\|v\left(x^k\right)\right\|\geq \gamma\text{ for all } k\geq 0.\]
		By the definition of the conjugate direction $d^k$ given in \eqref{conjugate direction in algorithm}, and from item (ii) of Lemma \ref{inequality lemma} with $p=v\left(x^k\right)$, $q=\beta_k^{mDY}\left\|d^{k-1}\right\|$, and $\alpha=\tfrac{1}{\sqrt{2\left(\zeta^2-1\right)}}$, we obtain that 
		\begin{align*}
			 \left\|d^k\right\|^2\leq \left(\left\|v\left(x^k\right)\right\|+\beta_k^{mDY}\left\|d^{k-1}\right\|\right)^2\leq\tfrac{\zeta^2}{\zeta^2-1}\left\|v\left(x^k\right)\right\|^2+\left(\beta_k^{mDY}\right)^2\zeta^2\left\|d^{k-1}\right\|^2.
		\end{align*}
		By \eqref{proof of mDY theorem eq 3}, we obtain that 
			\begin{align*}
			\left\|d^k\right\|^2\leq \tfrac{\zeta^2}{\zeta^2-1}\left\|v\left(x^k\right)\right\|^2+\tfrac{\left[\psi\circ\phi_{x^k}\left(d^k\right)\right]^2}{\left[\psi\circ\phi_{x^{k-1}}\left(d^{k-1}\right)\right]^2}\left\|d^{k-1}\right\|^2.
		\end{align*}
		By \eqref{proof of mDY theorem eq 1}, we get
		\begin{align*}
			\tfrac{\left\|d^k\right\|^2}{\left[\psi\circ\phi_{x^k}\left(d^k\right)\right]^2}\leq \tfrac{\zeta^2}{\zeta^2-1}\tfrac{\left\|v\left(x^k\right)\right\|^2}{\left[\psi\circ\phi_{x^k}\left(d^k\right)\right]^2}+\tfrac{\left\|d^{k-1}\right\|^2}{\left[\psi\circ\phi_{x^{k-1}}\left(d^{k-1}\right)\right]^2}\leq \tfrac{\left(\zeta+\sigma\right)^2}{\zeta^2-1}\tfrac{\left\|v\left(x^k\right)\right\|^2}{\left[\psi\circ\phi_{x^k}\left(v\left(x^k\right)\right)\right]^2}+\tfrac{\left\|d^{k-1}\right\|^2}{\left[\psi\circ\phi_{x^{k-1}}\left(d^{k-1}\right)\right]^2}.
		\end{align*}
		Since $0<\gamma^2\leq \left\|v\left(x^k\right)\right\|^2\leq -2~\psi\circ\phi_{x^k}\left(v\left(x^k\right)\right)$, we get 
		\begin{align*}
			\tfrac{\left\|d^k\right\|^2}{\left[\psi\circ\phi_{x^k}\left(d^k\right)\right]^2}\leq  \tfrac{4\left(\zeta+\sigma\right)^2}{\left(\zeta^2-1\right)\gamma^2}+\tfrac{\left\|d^{k-1}\right\|^2}{\left[\psi\circ\phi_{x^{k-1}}\left(d^{k-1}\right)\right]^2}. 
		\end{align*}
		Applying repeatedly, we get 
		\begin{align*}
			\tfrac{\left\|d^k\right\|^2}{\left[\psi\circ\phi_{x^k}\left(d^k\right)\right]^2}\leq  ak+b,\text{ where } a:= \tfrac{4\left(\zeta+\sigma\right)^2}{\left(\zeta^2-1\right)\gamma^2}>0 \text{ and } b:=\tfrac{\left\|v\left(x^0\right)\right\|^2}{\left[\psi\circ\phi_{x^{0}}\left(v\left(x^0\right)\right)\right]^2}>0.
		\end{align*}
		Therefore, we get 
		\begin{align*}
		\underset{k\geq0}{\sum}\tfrac{\left[\psi\circ\phi_{x^k}\left(d^k\right)\right]^2}{\left\|d^k\right\|^2}\geq	\underset{k\geq0}{\sum}\tfrac{1}{ak+b}=\infty,
		\end{align*}
		which contradicts to the Zoutendijk condition \eqref{Zountendjk condition}. Hence, the proof is completed.
\end{proof}

\section{Numerical Experiments}\label{Numerical Experiments}
This section presents the numerical behavior of Algorithm \ref{Algorithm} on several benchmark problems discussed in \cite{mondal2025steepest}. All computations were carried out using MATLAB 2023a. The experiments were performed on a system equipped with an Intel(R) Core(TM) i5-1035G1 (10th generation) processor running at 1.00 GHz–1.19 GHz, along with 8 GB of RAM.

The implementation of Algorithm \ref{Algorithm} in MATLAB uses the following settings:
\begin{itemize}
	\item The initial vector $x^0$ is generated randomly within the permissible bounds of each test instance, where the limits of the variables follow those reported in \cite{mondal2025steepest}. The MATLAB function ``rand" is used to sample this starting point.
	\item For every test problem, the parameters associated with the Wolfe line search are fixed as $\rho=0.001$ and $\sigma=0.1$. The step length $t_k$ at iteration $k$ is selected using the strong Wolfe conditions specified in \eqref{Strong Wolfe condition in algorithm}. 
	\item The quantities $v\left(x^k\right)$ and $\xi\left(x^k\right)$ at each iteration are computed by solving the optimization subproblem in \eqref{equivalent constrained problem} through MATLAB's ``quadprog" optimization toolbox.
	\item The termination of Algorithm \ref{Algorithm} follows the criterion $\xi\left(x^k\right)>-\epsilon$, where the tolerance level is chosen $\epsilon=10^{-6}$.
	\item For each iteration, $k$, we take the algorithmic parameter $\beta_k$ as $0.98\beta_k^{FR}$ for the FR method, $0.89\beta_k^{CD}$ for the CD method, and $0.81\beta_k^{DY}$ for the DY method, respectively. For the mDY method, we take $\beta_k:=\beta_k^{mDY}$ with $\zeta=1.03$.
\end{itemize}

With these considerations of parameter values, we apply Algorithm \ref{Algorithm} to each test problem given in \cite{mondal2025steepest}. To show the performance of Algorithm \ref{Algorithm} and comparison with the existing SD method \cite{mondal2025steepest} for an MIOP, we compute the iteration number and CPU time in seconds, which are given in Table \ref{performance table on iteration} and Table \ref{performance table on CPU}, respectively. Taking 100 randomly chosen initial points, we compute $\left(\min,\text{mean},\max\right)$ for both of the iteration numbers and CPU time for each test problem. Further, we depict the performance profile in Figure \ref{figure:performance profile} from the perspective of Dolan-Mor{\'e} \cite{dolan2002benchmarking} performance profile to compare the performance of Algorithm \ref{Algorithm} with the algorithm of the SD method \cite{mondal2025steepest}. Let $\mathcal{S}$ and $\mathcal{P}$ be the set of solvers and the set of problems, respectively. In addition, we assume that $N_s$ and $N_p$ are the number of solvers and the number of problems, respectively. We are interested in using average iteration numbers and average CPU time as performance measures. For each problem $p$ and solver $s$, we define \[I_{p,s}:=\text{ average number of iterations required to solve problem }p \text{ by solver }s\] and \[T_{p,s}:=\text{ average computing time required to solve problem }p \text{ by solver }s.\]
We compare the performance on problem $p$ by solver $s$ with the best performance by any solver on this problem, i.e., we define the performance ratio for average iteration numbers and average CPU time by 
\[R_{p,s}^I:=\tfrac{I_{p,s}}{\min\left\{I_{p,s}:s\in{\mathcal{S}}\right\}} \text{ and }R_{p,s}^T:=\tfrac{T_{p,s}}{\min\left\{T_{p,s}:s\in{\mathcal{S}}\right\}}.\] 
The performance profile $F_I:{\mathbb{R}}\to[0,1]$ measured by average iteration numbers and the performance profile $F_T:{\mathbb{R}}\to[0,1]$ measured by average CPU time are defined by \[F_I\left(z\right):=\tfrac{1}{N_p}\text{ size }\left\{p\in{\mathcal{P}}:R_{p,s}^I\leq z\right\} \text{ and } F_T\left(z\right):=\tfrac{1}{N_p}\text{ size }\left\{p\in{\mathcal{P}}:R_{p,s}^T\leq z\right\}\text{ for all } z\in {\mathbb{R}}.\]
Note that $F_I(z)$ and $F_T(z)$ are the probabilities for solver $s\in {\mathcal{S}}$ that the performance ratios $R_{p,s}^I$ and $R_{p,s}^T$ are within a factor $z\in{\mathbb{R}}$ of the best possible ratios, respectively. The functions $F_I$ and $F_T$ are the cumulative distribution functions for the performance ratios $R_{p,s}^I$ and $R_{p,s}^T$, respectively, which are depicted in Figure \ref{figure:performance profile}. 
	{\footnotesize
	\renewcommand{\arraystretch}{1.05}
	
	\begin{longtable}{@{} l *{5}{l} @{}}
		\caption{Performance of Algorithm \ref{Algorithm} in terms of Iteration Number \label{performance table on iteration}} \\
		
		\toprule
		\multirow{2}{*}{\text{Problem}} &
		\text{SD} & \text{FR} & \text{CD} & \text{DY} & \text{mDY} \\
		& {\footnotesize (min, mean, max)} & {\footnotesize (min, mean, max)} & {\footnotesize (min, mean, max)} & {\footnotesize (min, mean, max)} & {\footnotesize (min, mean, max)} \\
		\midrule
		\endfirsthead
		
		\toprule
		\multirow{2}{*}{\text{Problem}} &
		\text{SD} & \text{FR} & \text{CD} & \text{DY} & \text{mDY} \\
		& {\footnotesize (min, mean, max)} & {\footnotesize (min, mean, max)} & {\footnotesize (min, mean, max)} & {\footnotesize (min, mean, max)} & {\footnotesize (min, mean, max)} \\
		\midrule
		\endhead
		
		\midrule \multicolumn{6}{r}{\footnotesize Continued on next page} \\
		\endfoot
		
		\bottomrule
		\endlastfoot
		
		I-BK1 & (0, 2.49, 5) & (0, 3.07, 13) & (0, 3.21, 9) & (0, 2.83, 7) & (0, 3.45, 11) \\
		I-VU2 & (0, 27.74, 69) & (0, 228.32, 667)
		 & (0, 192.80, 571) & (0, 125.57, 313) & (0, 394.51, 1206) \\
		I-CH &  (1, 5.87, 11) & (1, 28.01, 106)
		 & (1, 13.07, 32) & (1, 10.75, 22) & (1, 21.20, 75) \\
		I-FON &   (0, 10.75, 45) & (0, 112.84, 493)
		 & (0, 46.24, 329) & (0, 48.80, 199) & (0, 114.55, 748) \\
		 I-KW2 &   (0, 19.02, 67) & (0, 82.30, 648)
		  & (0, 124.78, 1621) & (0, 86.09, 609) & (0, 135.84, 820) \\
		 I-Far1 &   (0, 4.15, 57) & (0, 18.92, 189) & (0, 21.46, 631) & (0, 30.97, 352) & (0, 20.21, 259) \\
		 I-Hil1 &   (0, 10.86, 111) & (0, 46.01, 340) & (0, 47.74, 480) & (0, 25.91, 164) & (0, 73.22, 684) \\
		 I-PNR &    (0, 6.43, 22) & (0, 43.94, 126) & (0, 23.65, 167) & (0, 17.73, 86) & (0, 44.90, 146) \\
		 I-Deb &    (0, 224.71, 20007) & (0, 13.16, 85)
		  & (0, 8.74, 79) & (0, 15.62, 105) & (0, 10.39, 92)\\
		 I-SD &   (0, 3.64, 8) & (0, 6.00, 20)
		 & (0, 4.48, 14) & (0, 4.09, 11) & (0, 5.02, 15)\\
    	I-IKK1 &   (0, 6.28, 349) & (0, 21.56, 880)
		 & (0, 35.81, 1251) & (0, 16.06, 488) & (0, 277.85, 8227)\\
		 I-VFM1 &    (0, 1.52, 4) & (0, 1.97, 6)
		 & (0, 1.65, 6) & (0, 1.46, 5) & (0, 1.88, 6)\\
		 I-MHHM2 &    (0, 3.37, 6) & (1, 8.06, 22)
		 & (1, 7.13, 14) & (0, 6.11, 12) & (0, 7.92, 21)\\
I-Viennet & (0, 1.88, 100) & (0, 3.23, 103) & (0, 1.29, 37) & (0, 1.65, 36) & (0, 4.09, 88)\\
I-AP1 &    (0, 93.98, 1569) & (0, 343.20, 3408)
& (0, 167.00, 1571) & (1, 111.50, 1095) & (0, 28.40, 237)\\
	I-MOP7 &    (25, 67.78, 125) & (2, 7.56, 15)
& (2, 7.99, 14) & (2, 8.08, 15) & (2, 8.14, 15)\\
I-VFM2 &    (0, 17.82, 34) & (0, 35.47, 2703)
& (0, 5.50, 87) & (0, 5.04, 58) & (0, 14.74, 351)\\
I-TR1 &    (8, 8.64, 10) & (4, 4.00, 4) & (5, 5.00, 5) & (8, 9.60, 12) & (8, 10.03, 13)\\
I-AP4 &    (0, 141.71, 2301) & (0, 34.88, 301)
& (0, 21.56, 209) & (0, 253.04, 5989) & (0, 328.60, 7345)\\
I-Comet &    (0, 80.01, 131) & (0, 117.40, 278)
& (0, 85.20, 386) & (0, 294.80, 602) & (0, 425.00, 903)\\
	\end{longtable}
	
	\vspace{1cm}
	
	\begin{longtable}{@{} l *{5}{l} @{}}
		\caption{Performance of Algorithm \ref{Algorithm} in terms of CPU Time (s)\label{performance table on CPU}} \\
		
		\toprule
		\multirow{2}{*}{\text{Problem}} &
		\text{SD} & \text{FR} & \text{CD} & \text{DY} & \text{mDY} \\
		& {\footnotesize (min, mean, max)} & {\footnotesize (min, mean, max)} & {\footnotesize (min, mean, max)} & {\footnotesize (min, mean, max)} & {\footnotesize (min, mean, max)} \\
		\midrule
		\endfirsthead
		
		\toprule
		\multirow{2}{*}{\text{Problem}} &
		\text{SD} & \text{FR} & \text{CD} & \text{DY} & \text{mDY} \\
		& {\footnotesize (min, mean, max)} & {\footnotesize (min, mean, max)} & {\footnotesize (min, mean, max)} & {\footnotesize (min, mean, max)} & {\footnotesize (min, mean, max)} \\
		\midrule
		\endhead
		
		\midrule \multicolumn{6}{r}{\footnotesize Continued on next page} \\
		\endfoot
		
		\bottomrule
		\endlastfoot
		
		I-BK1 &  (0.02, 0.07, 0.12) & (0.02, 0.09, 0.30) & (0.02, 0.09, 0.21) & (0.02, 0.08, 0.17) & (0.02, 0.09, 0.24) \\
		IVU2  & (0.02, 0.56, 1.36) & (0.02, 5.07, 15.93) & (0.02, 5.27, 16.14) & (0.02, 3.03, 15.59) & (0.02, 8.80, 35.29) \\
		I-CH &  (0.04, 0.13, 0.28) & (0.04, 0.61, 2.25) & (0.04, 0.30, 0.71) & (0.04, 0.25, 0.50) & (0.04, 0.47, 1.58) \\
		I-FON &   (0.02, 0.23, 0.89) & (0.02, 2.46, 10.62) & (0.02, 1.00, 6.95) & (0.02, 1.07, 4.27) & (0.02, 2.80, 16.35) \\
		I-KW2 &   (0.02, 0.39, 1.31) & (0.02, 1.87, 14.49) & (0.02, 3.25, 36.36) & (0.02, 1.93, 13.97) & (0.02, 3.21, 18.18) \\
		I-Far1 &    (0.02, 0.10, 1.10) & (0.02, 0.43, 4.07) & (0.02, 0.50, 14.19) & (0.02, 0.74, 7.65) & (0.02, 0.49, 5.98) \\
		I-Hil1 &    (0.02, 0.23, 2.15) & (0.02, 1.02, 7.32) & (0.02, 1.05, 10.43) & (0.02, 0.58, 3.50) & (0.02, 1.61, 14.81) \\
		I-PNR &   (0.02, 0.14, 0.44) & (0.02, 0.97, 2.76) & (0.02, 0.53, 3.58) & (0.02, 0.40, 1.91) & (0.02, 1.00, 3.10) \\
		I-Deb &    (0.02, 4.63, 413.67) & (0.02, 0.30, 1.86) & (0.02, 0.20, 1.69) & (0.02, 0.60, 4.95) & (0.02, 0.24, 2.01) \\
		I-SD &    (0.02, 0.09, 0.17) & (0.02, 0.21, 0.82)
		& (0.02, 0.12, 0.36) & (0.02, 0.11, 0.25) & (0.02, 0.13, 0.38)\\
		I-IKK1 &    (0.02, 0.14, 6.86) & (0.02, 0.48, 18.73)
		& (0.02, 0.78, 26.26) & (0.02, 0.36, 10.39) & (0.02, 6.44, 193.81)\\
		 I-VFM1 &    (0.02, 0.05, 0.13) & (0.02, 0.14, 6.60)
		& (0.02, 0.07, 0.19) & (0.02, 0.05, 0.18) & (0.02, 0.07, 0.21)\\
		I-MHHM2 &    (0.02, 0.09, 0.16) & (0.04, 0.19, 0.49)
		& (0.04, 0.19, 0.52) & (0.02, 0.15, 0.28) & (0.02, 0.18, 0.43)\\
		I-Viennet &    (0.02, 0.06, 1.96) & (0.02, 0.09, 2.11)
		& (0.02, 0.05, 0.73) & (0.02, 0.06, 0.86) & (0.02, 0.12, 2.10)\\
		I-AP1 &    (0.02, 1.81, 30.15) & (0.02, 7.16, 70.86)
		& (0.02, 3.98, 36.89) & (0.04, 2.23, 21.73) & (0.02, 0.58, 4.71)\\
		I-MOP7 &    (0.48, 1.30, 2.43) & (0.06, 0.17, 0.31)
		& (0.06, 0.18, 0.30) & (0.06, 0.18, 0.31) & (0.06, 0.18, 0.31)\\
	I-VFM2 &    (0.02, 0.45, 2.19) & (0.02, 0.73, 53.82)
		& (0.02, 0.13, 1.77) & (0.02, 0.12, 1.20) & (0.02, 0.31, 6.95)\\
		I-TR1 &    (0.02, 0.20, 0.39) & (0.10, 0.10, 0.13)
		&(0.12, 0.13, 0.15) & (0.18, 0.22, 0.27) & (0.18, 0.23, 0.29)\\
		I-AP4 &    (0.02, 2.92, 55.34) & (0.02, 0.72, 6.06)
		& (0.02, 0.46, 4.22) & (0.02, 5.10, 120.28) & (0.02, 6.64, 147.83)\\
		I-Comet &    (0.02, 1.60, 2.94) & (0.02, 2.41, 5.85)
		& (0.02, 1.74, 7.79) & (0.02, 5.97, 12.12) & (0.02, 8.54, 18.11)\\
		
	\end{longtable}
}
	
		\begin{figure}[htbp]
		\begin{subfigure}[t]{0.45\textwidth}
			\includegraphics[width=\linewidth]{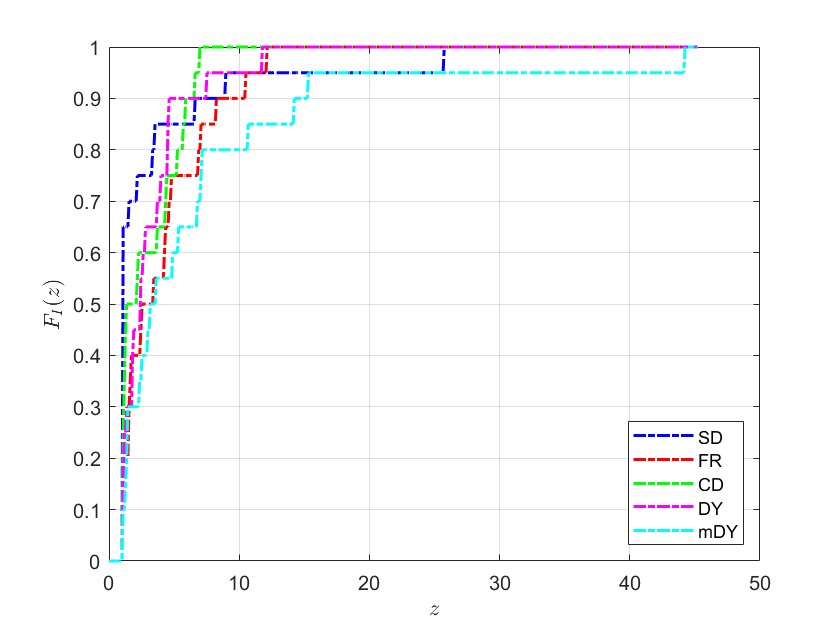}
			\caption{Performance profile of the methods: SD, FR, CD, DY, and mDY measured by average iteration numbers}
			\label{iteration}
		\end{subfigure}\hfill
		\begin{subfigure}[t]{0.45\textwidth}
			\includegraphics[width=\linewidth]{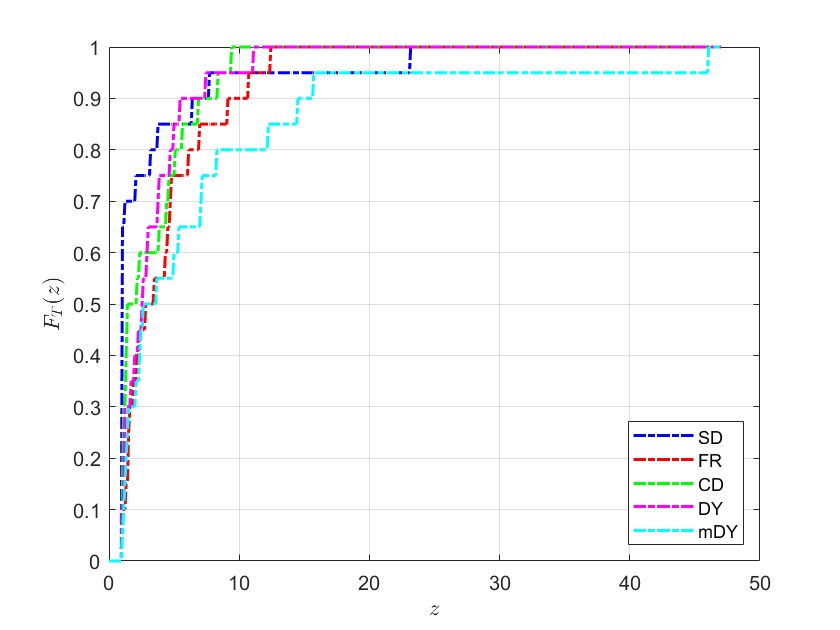}
			\caption{Performance profile of the methods: SD, FR, CD, DY, and mDY measured by average CPU time}
			\label{CPU time}
		\end{subfigure}
		\caption{Performance profile.}
		\label{figure:performance profile}
	\end{figure}
	From Table \ref{performance table on iteration} and Table \ref{performance table on CPU}, we have the following observations:
	\begin{itemize}
		\item Among FR, CD, DY, and mDY methods, the FR method gives the best performance for the test problems: I-KW2, I-Far1, I-MOP7, and I-TR1.
		\item Among FR, CD, DY, and mDY methods, the CD method gives the best performance for the test problems: I-FON, I-Deb, I-Viennet, I-AP4, and I-Commet.
		\item The DY method gives the best performance for the test problems: I-BK1, I-VU2, I-CH, I-Hil1, I-PNR, I-SD, I-IKK1, I-VFM1, I-MHHM2, and I-VFM2.
		\item The mDY method gives the best performance for the test problem I-AP1.
	\end{itemize}
	For a particular test problem, the method that gives the best performance among the methods FR, CD, DY, and mDY is used. For example, we use the DY method for the test problem I-VU2.

	\bigskip
	\noindent
	For each biobjective benchmark problem described in \cite{mondal2025steepest}, the corresponding attainable objective region is depicted in Figure \ref{figure:biobjective}. Additionally, using five randomly selected starting points, the images of the optimal solutions $G\left(x^\star\right)$ are also shown within these regions. The light blue shaded portion indicates the collection of all achievable objective rectangles, while each magenta colored rectangle highlights the images of the optimal solutions $G\left(x^\star\right)$. To approximate the feasible objective domain, $5000$ random points are sampled from the variable bounds of each test instance. Each sample $x$ produces a rectangle $G(x)$, and the union of all such rectangles represents the overall objective region, i.e.,
	\[\text{objective feasible region}:=\underset{lb\leq x \leq ub}{\bigcup}G(x).\]
	The black dot marks the center of $G\left(x^0\right)$, the blue dot marks the center of $G\left(x^\star\right)$, and the two are connected with a magenta segment.
	
	\bigskip
	\noindent
	For triobjective cases, direct visualization of the full feasible objective region is impractical. Therefore, for a single randomly selected initial point $x^0$, we only display the image of the starting point $G\left(x^0\right)$
	and the corresponding image of the optimal point $G\left(x^\star\right)$ in Figure \ref{figure:triobjective}. Here, the cyan colored cube corresponds to $G\left(x^0\right)$, and the magenta colored cube represents $G\left(x^\star\right)$. The magenta colored path joining the black and blue dots illustrates the progression of the centers of the generated sequence $\left\{G\left(x^k\right)\right\}$ produced by Algorithm \ref{Algorithm}.
	
	\begin{figure}[htbp]
		\begin{subfigure}[t]{0.33\textwidth}
			\includegraphics[width=\linewidth]{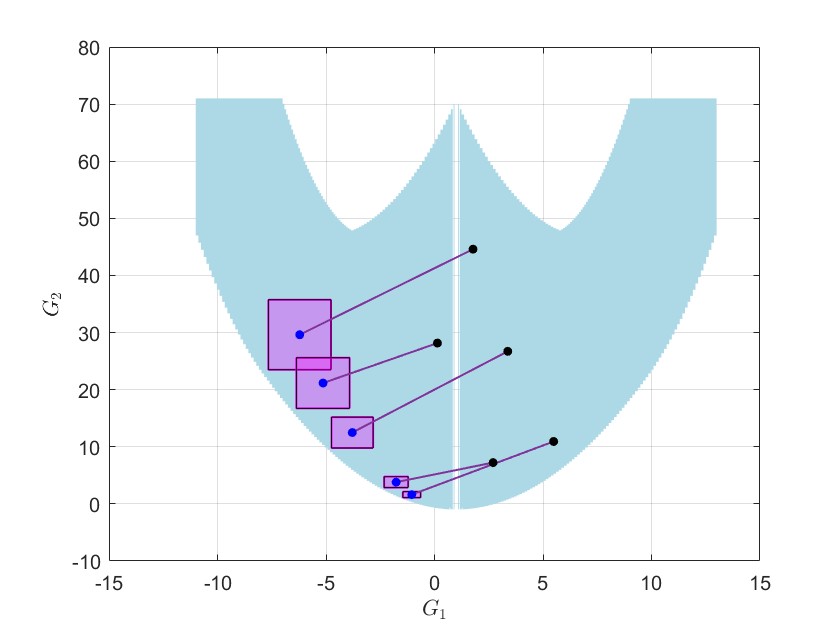}
			\caption{I-VU2}
			\label{fig:I-VU2}
		\end{subfigure}\hfill
		\begin{subfigure}[t]{0.33\textwidth}
			\includegraphics[width=\linewidth]{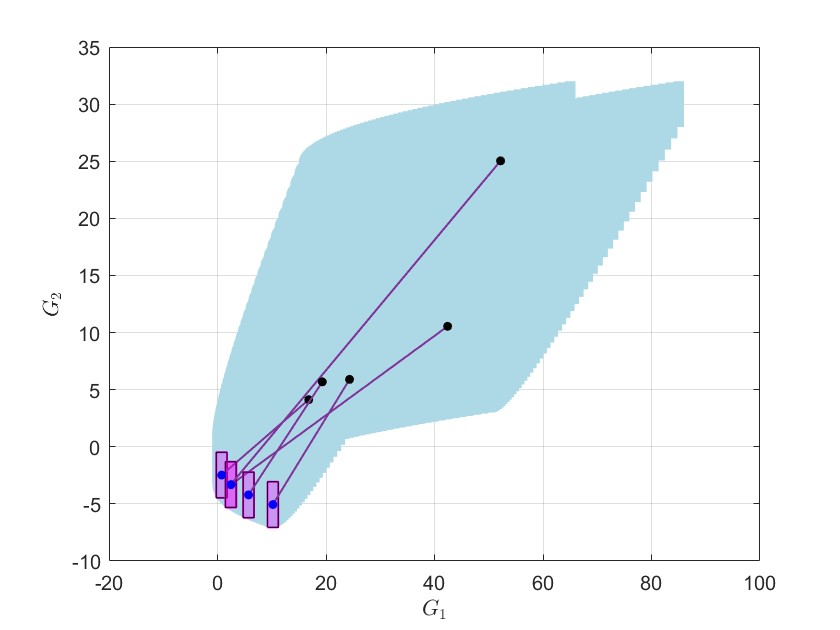}
			\caption{I-CH}
			\label{fig:I-CH}
		\end{subfigure}\hfill
		\begin{subfigure}[t]{0.33\textwidth}
			\includegraphics[width=\linewidth]{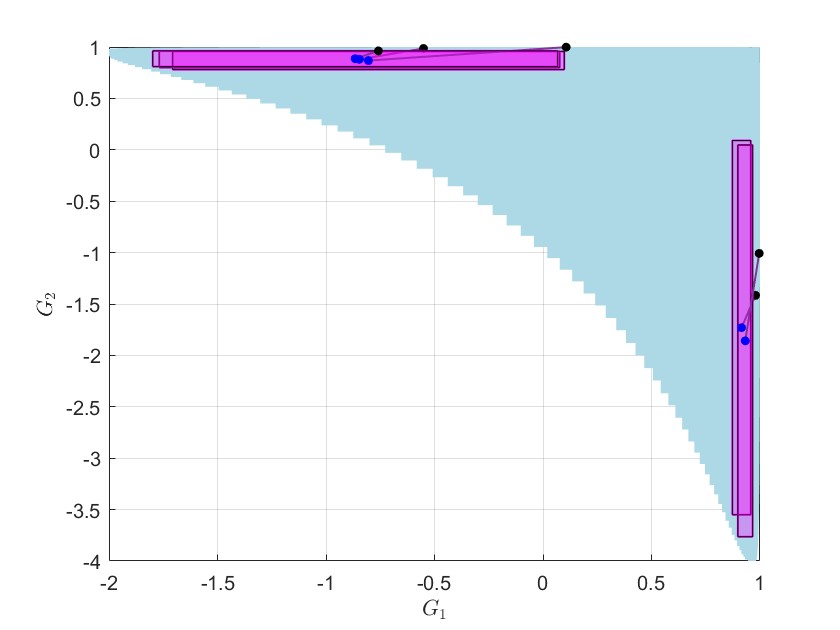}
			\caption{I-FON}
			\label{fig:I-FON}
		\end{subfigure}
		
		\begin{subfigure}[t]{0.33\textwidth}
			\includegraphics[width=\linewidth]{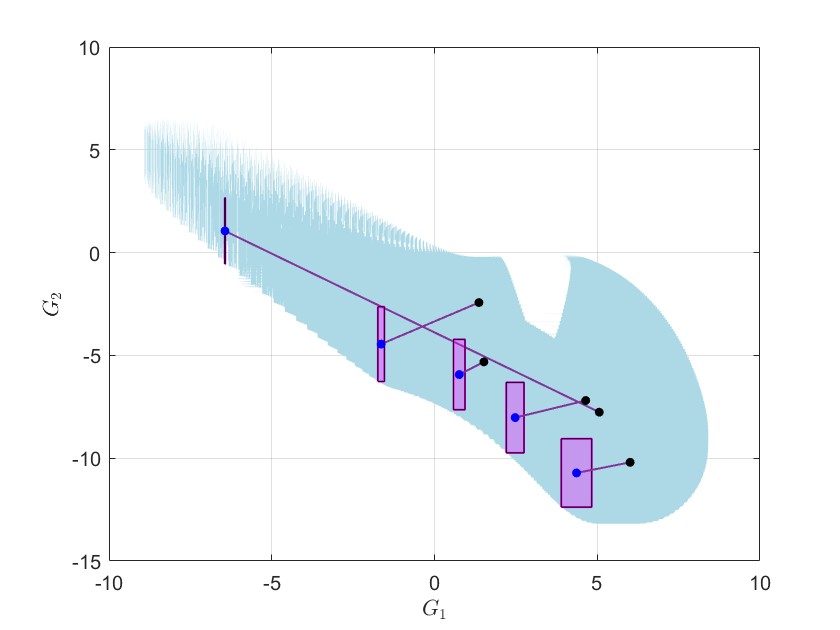}
			\caption{I-KW2}
			\label{fig:I-KW2}
		\end{subfigure}\hfill
		\begin{subfigure}[t]{0.33\textwidth}
			\includegraphics[width=\linewidth]{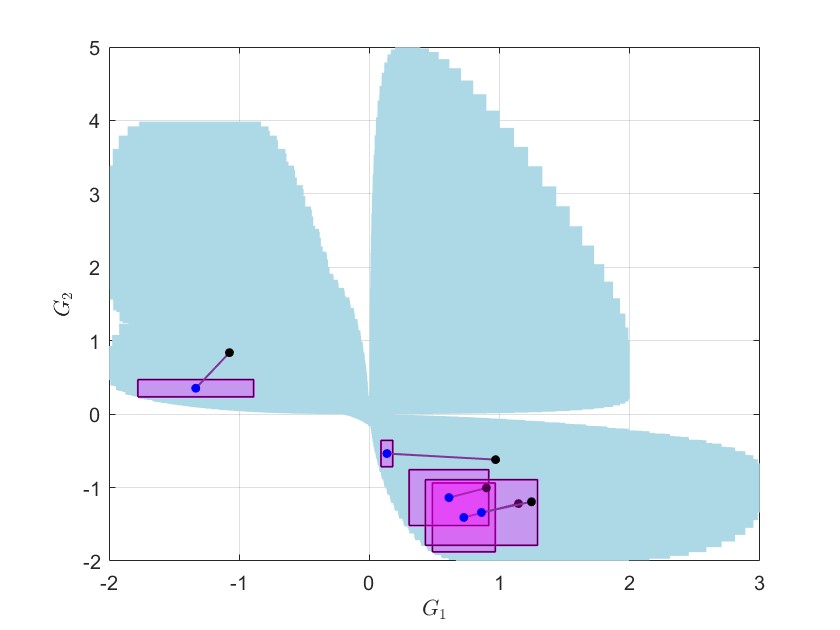}
			\caption{I-Far1}
			\label{fig:I-Far1}
		\end{subfigure}\hfill
		\begin{subfigure}[t]{0.33\textwidth}
			\includegraphics[width=\textwidth]{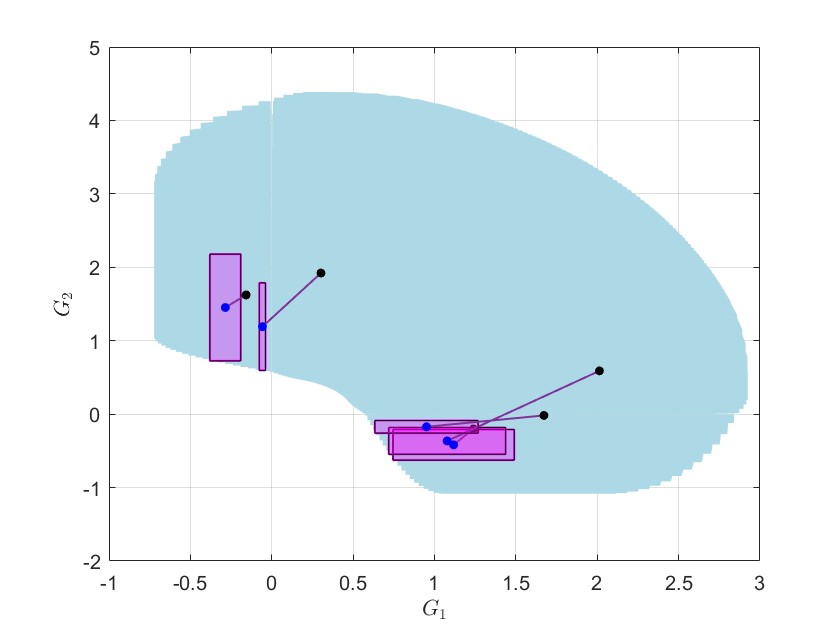}
			\caption{I-Hil1}
			\label{fig:I-Hil1}
		\end{subfigure}
		
		\begin{subfigure}[t]{0.33\textwidth}
			\includegraphics[width=\linewidth]{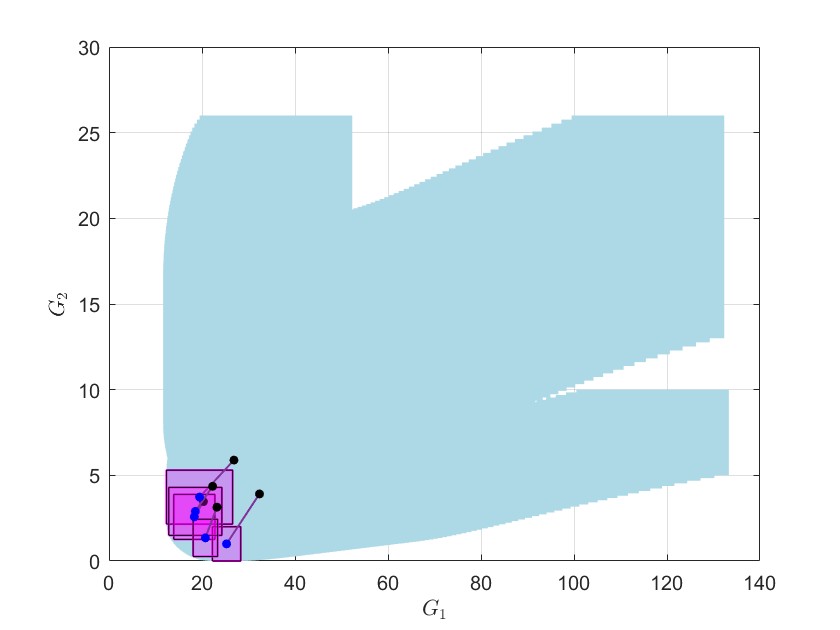}
			\caption{I-PNR}
			\label{fig:I-PNR}
		\end{subfigure}\hfill
		\begin{subfigure}[t]{0.33\textwidth}
			\includegraphics[width=\linewidth]{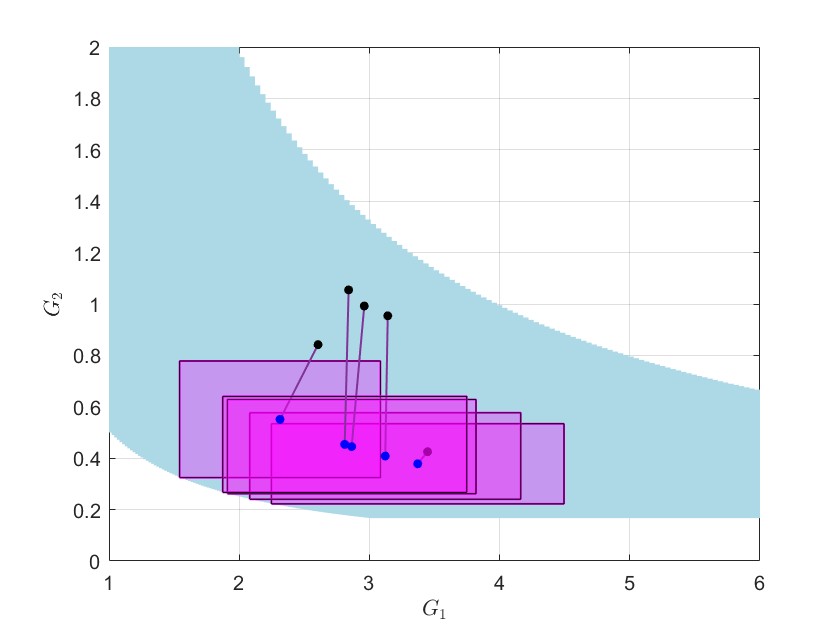}
			\caption{I-Deb}
			\label{fig:I-Deb}
		\end{subfigure}\hfill
		\begin{subfigure}[t]{0.33\textwidth}
			\includegraphics[width=\linewidth]{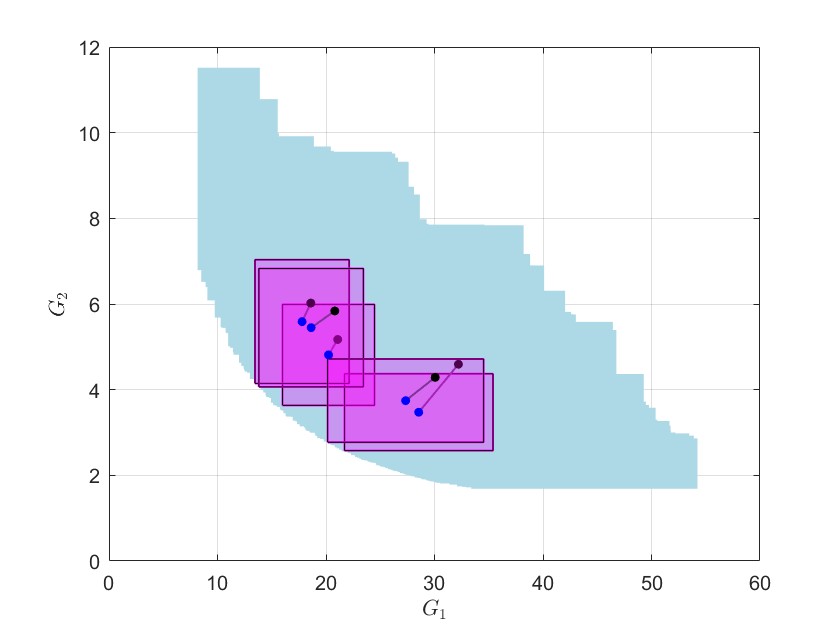}
			\caption{I-SD}
			\label{fig:I-SD}
		\end{subfigure}
		
		\caption{For five randomly selected starting points, the locations of $G\left(x^\star\right)$ within the attainable objective region for each biobjective test problem.}
		\label{figure:biobjective}
	\end{figure}
	
	\begin{figure}[htbp]
		\begin{subfigure}[t]{0.33\textwidth}
			\includegraphics[width=\linewidth]{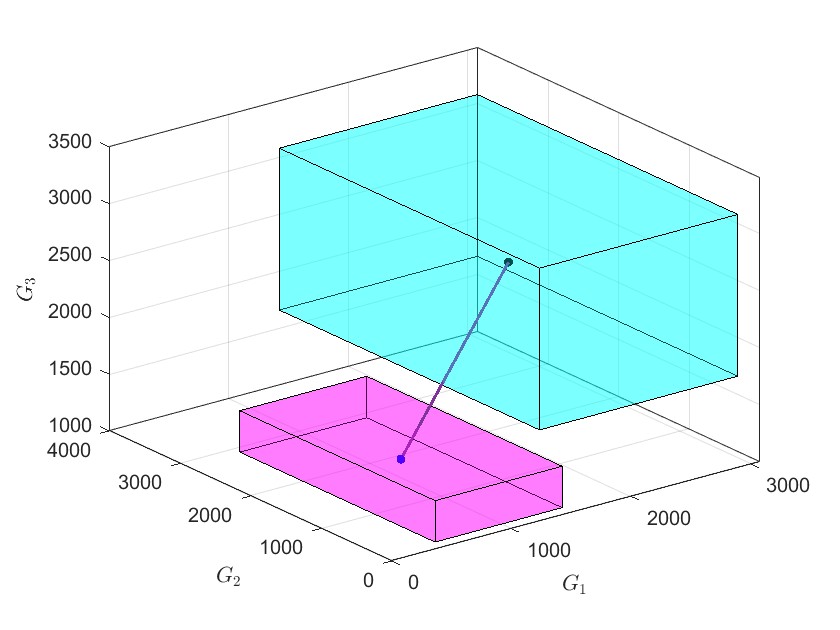}
			\caption{I-IKK1}
			\label{fig:I-IKK1}
		\end{subfigure}\hfill
		\begin{subfigure}[t]{0.33\textwidth}
			\includegraphics[width=\linewidth]{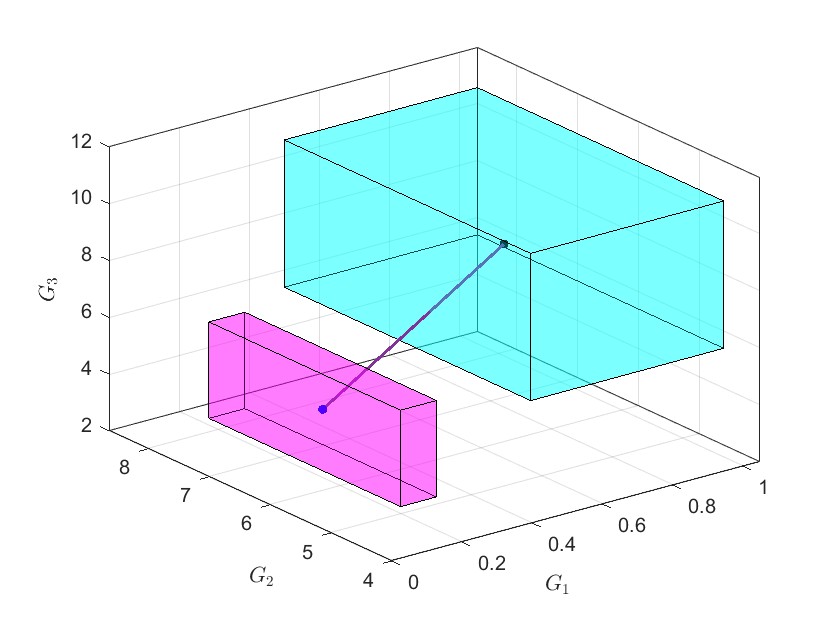}
			\caption{I-VFM1}
			\label{fig:I-VFM1}
		\end{subfigure}\hfill
		\begin{subfigure}[t]{0.33\textwidth}
			\includegraphics[width=\linewidth]{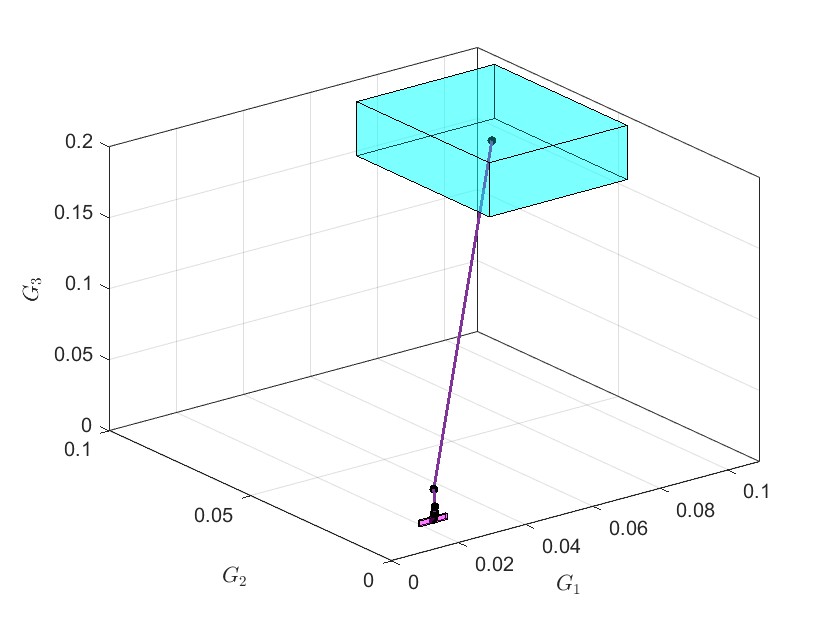}
			\caption{I-MHHM2}
			\label{fig:I-MHHM2}
		\end{subfigure}
		
		\begin{subfigure}[t]{0.33\textwidth}
			\includegraphics[width=\linewidth]{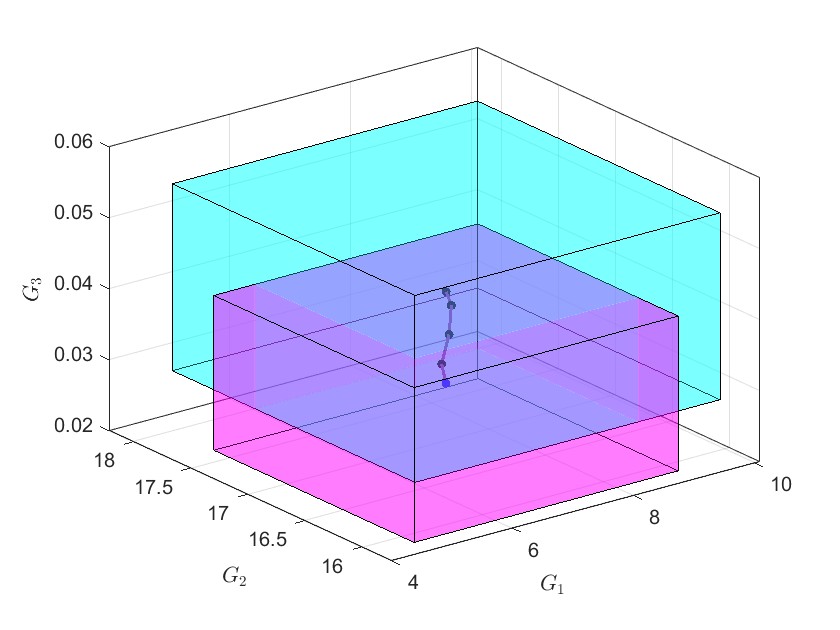}
			\caption{I-Viennet}
			\label{fig:I-Viennet}
		\end{subfigure}\hfill
		\begin{subfigure}[t]{0.33\textwidth}
			\includegraphics[width=\linewidth]{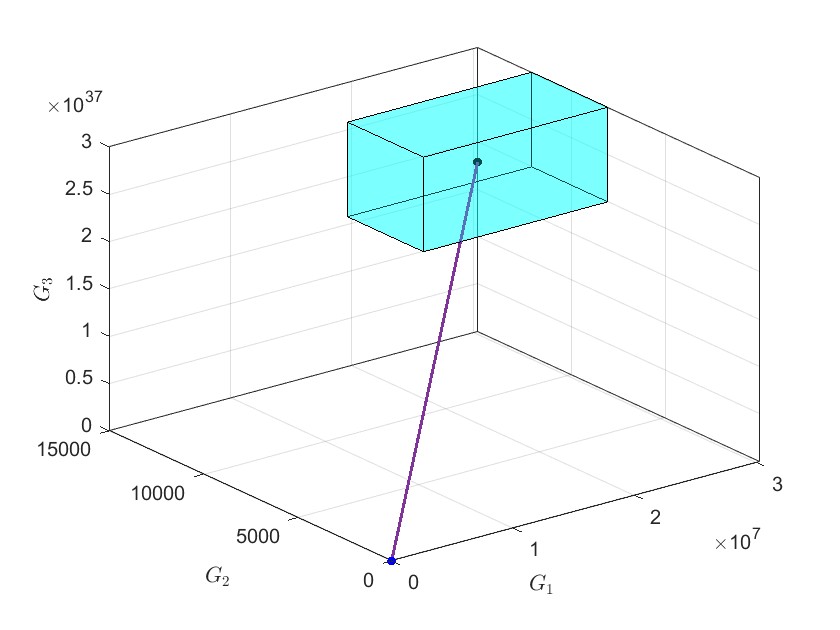}
			\caption{I-AP1}
			\label{fig:I-AP1}
		\end{subfigure}\hfill
		\begin{subfigure}[t]{0.33\textwidth}
			\includegraphics[width=\textwidth]{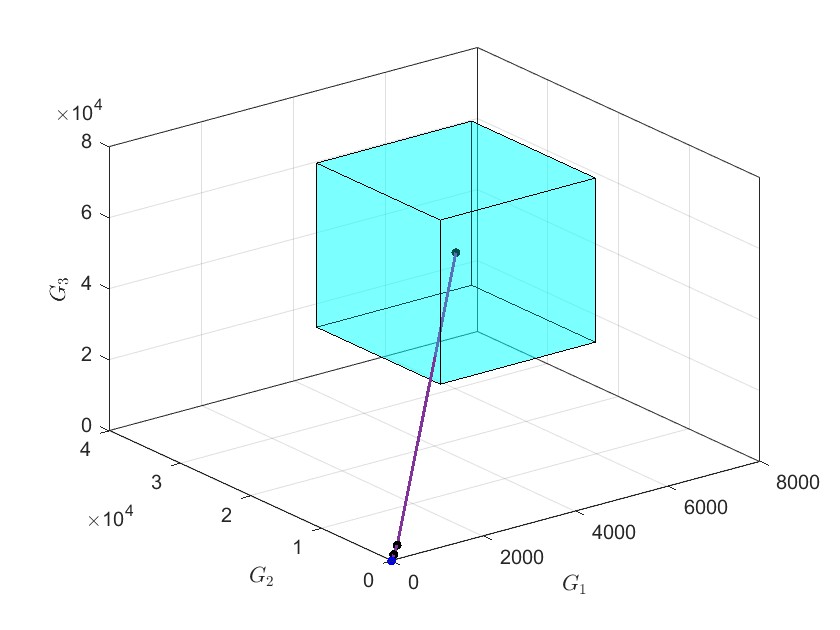}
			\caption{I-MOP7}
			\label{fig:I-MOP7}
		\end{subfigure}
		
		\begin{subfigure}[t]{0.33\textwidth}
			\includegraphics[width=\linewidth]{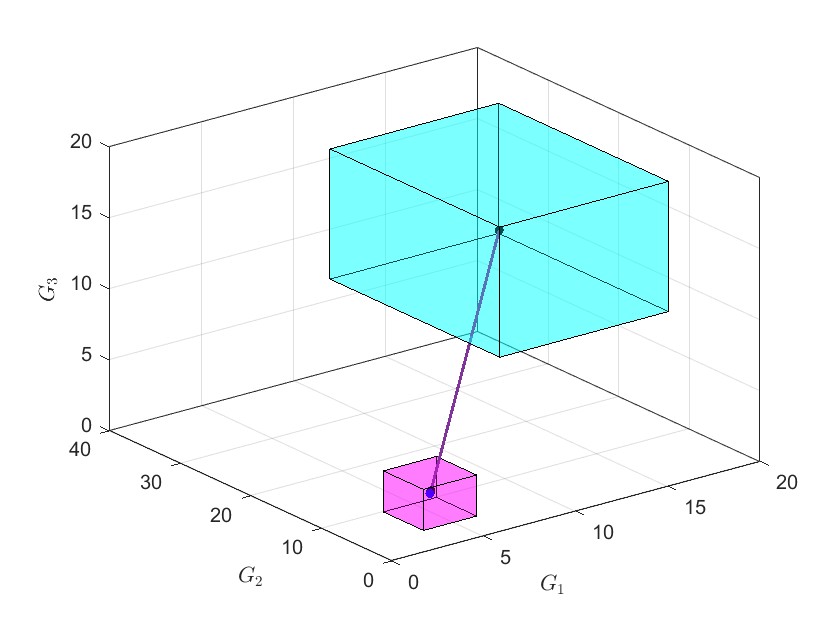}
			\caption{I-VFM2}
			\label{fig:I-VFM2}
		\end{subfigure}\hfill
		\begin{subfigure}[t]{0.33\textwidth}
			\includegraphics[width=\linewidth]{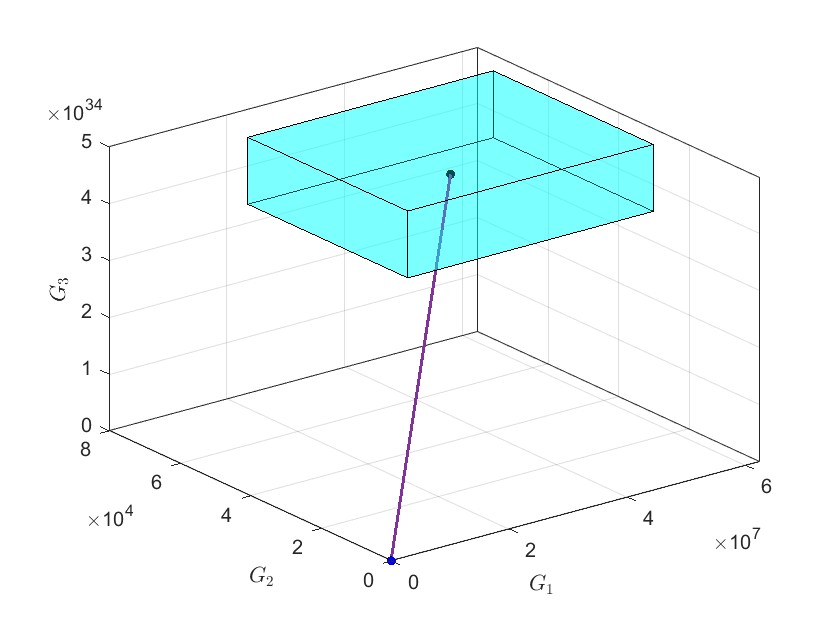}
			\caption{I-AP4}
			\label{fig:I-AP4}
		\end{subfigure}\hfill
		\begin{subfigure}[t]{0.33\textwidth}
			\includegraphics[width=\linewidth]{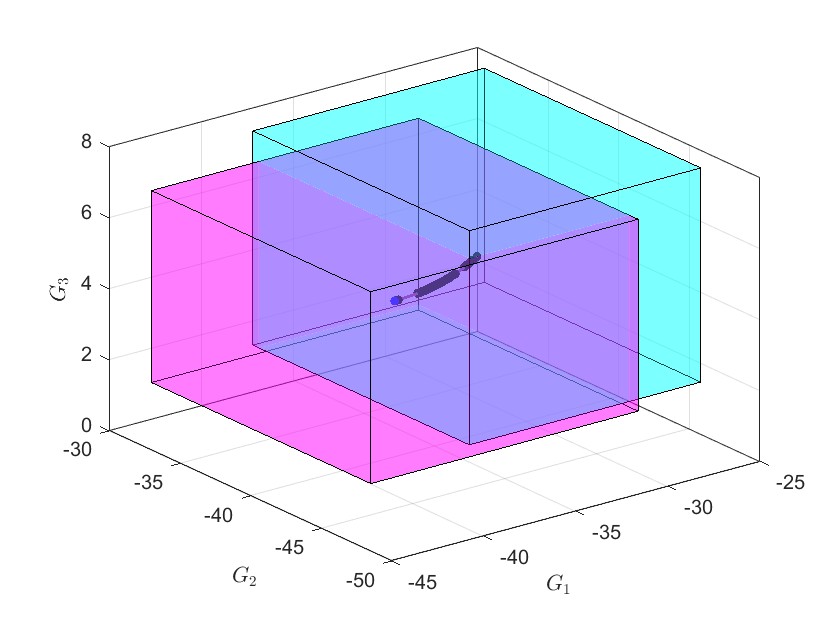}
			\caption{I-Comet}
			\label{fig:I-Comet}
		\end{subfigure}
		
		\caption{For a randomly selected starting point, the location of $G\left(x^0\right)$ and $G\left(x^\star\right)$ of the triobjective test problems.}
		\label{figure:triobjective}
	\end{figure}
\clearpage

\section{Conclusion and Future Directions}\label{Conclusion and Future Directions}
In this paper, we have studied the nonlinear conjugate gradient method for an unconstrained MIOP. First, we have defined the standard Wolfe conditions and the strong Wolfe conditions. Further, we have shown the existence of an interval of step length that satisfies the standard Wolfe conditions and the strong Wolfe conditions (Proposition \ref{existence of an interval of steplength}). Based on this line search, we have proposed an algorithm of the nonlinear conjugate method to find a Pareto critical point (Algorithm \ref{Algorithm}). To study the convergence analysis of Algorithm \ref{Algorithm}, we have derived the result related to the Zoutendijk conditions (Proposition \ref{Zountendijk condition proposition}). Based on this result, we have proved the global convergence result of Algorithm \ref{Algorithm} without imposing an explicit restriction on the algorithmic parameter $\beta_k$ (Theorem \ref{General convergence theorem}). Further, we have studied the global convergence results of Algorithm \ref{Algorithm} for different choices of the algorithmic parameter $\beta_k$. We have used four different choices of the algorithmic parameter: $\beta_k^{FR}$, $\beta_k^{CD}$, $\beta_k^{DY}$, and $\beta_k^{mDY}$. For each variant of the algorithmic parameter $\beta_k$, we have proved the global convergence results of Algorithm \ref{Algorithm} (Theorem \ref{global convergence thm for FR}, Theorem \ref{global convergence thm for CD parameter}, Theorem \ref{global convergence thm for DY parameter}, and Theorem \ref{mDY thm}). Finally, we have tested the proposed Algorithm \ref{Algorithm} through some test problems refer to \cite{mondal2025steepest}. In addition, we have shown the performance of Algorithm \ref{Algorithm} and comparison with the existing SD method \cite{mondal2025steepest} for an MIOP from the perspective of Dolan-Mor{\'e} \cite{dolan2002benchmarking} performance profile. It has been observed that the DY method is the best for most of the test problems among the FR, CD, DY, and mDY methods.

In this paper, we have used the Wolfe line search procedure to find the step length. Further, we will focus on the Armijo-like line search to find the step length. Consequently, we will extend the method and the convergence results. In this study, we have not considered the other variants of the algorithmic parameter $\beta_k$, such as PRP and HS variants. For future research, we will aim to study the nonlinear conjugate gradient method for PRP and HS variants.

\bigskip

\noindent {\bf Acknowledgments} 
The authors acknowledge the Science and Engineering Research Board, India, for supporting this study through the Core Research Grant (CRG/2022/001347). \\

\noindent {\bf Data Availability}\\ 
No data was used for the research described in the article. \\ 

\noindent {\bf Disclosure Statement} \\ 
The authors do not have any conflicts of interest.

\end{document}